\documentclass{article}

\PassOptionsToPackage{numbers, compress}{natbib}
\usepackage[preprint]{neurips_2024}

\usepackage[utf8]{inputenc} 
\usepackage[T1]{fontenc}    
\usepackage{hyperref}       
\usepackage{url}            
\usepackage{booktabs} 
\usepackage{amsthm}
\usepackage{amsfonts}       
\usepackage{nicefrac}       
\usepackage{microtype}      
\usepackage{xcolor, comment}         

\usepackage{amsmath}
\usepackage{amssymb}
\usepackage{graphicx}
\usepackage{latexsym}
\usepackage{tabularx}
\usepackage[linesnumbered,ruled,vlined]{algorithm2e}
\usepackage{makecell}
\usepackage{multirow}
\usepackage{cleveref}
\usepackage{threeparttable}
\RequirePackage{endnotes}

\newcommand{\assign}{:=}

\newcommand{\nobracket}{}
\newcommand{\tmop}[1]{\ensuremath{\operatorname{#1}}}

\newenvironment{tmparmod}[3]{\begin{list}{}{\setlength{\topsep}{0pt}\setlength{\leftmargin}{#1}\setlength{\rightmargin}{#2}\setlength{\parindent}{#3}\setlength{\listparindent}{\parindent}\setlength{\itemindent}{\parindent}\setlength{\parsep}{\parskip}} \item[]}{\end{list}}

\newcounter{tmcounter}

\newtheorem{definition}{Definition}
\newtheorem{lemma}{Lemma}
\newtheorem{theorem}{Theorem}
\newtheorem{proposition}{Proposition}

\title{Restarted Primal-Dual Hybrid Conjugate Gradient Method for Large-Scale Quadratic Programming}

\author{Yicheng Huang\thanks{Shanghai University of Finance and Economics} \and
  Wanyu Zhang\footnotemark[1] \and
  Hongpei Li\footnotemark[1]
  \and Dongdong Ge\thanks{Shanghai Jiao Tong University} \and Huikang Liu\footnotemark[2] \and Yinyu Ye\thanks{Stanford University}
}

\begin{document}

\maketitle
\footnotetext[1]{Correspondence to: hkl1u@sjtu.edu.cn}

\begin{abstract}
Convex quadratic programming (QP) is an essential class of optimization problems with broad applications across various fields. Traditional QP solvers, typically based on simplex or barrier methods, face significant scalability challenges 
. In response to these limitations, recent research has shifted towards matrix-free first-order methods to enhance scalability in QP. Among these, the restarted accelerated primal-dual hybrid gradient (rAPDHG) method, proposed by \citep{lu2023practical}, has gained notable attention due to its linear convergence rate to an optimal solution and its straightforward implementation on Graphics Processing Units (GPUs). Building on this framework, this paper introduces a restarted primal-dual hybrid conjugate gradient (PDHCG) method, which incorporates conjugate gradient (CG) techniques to address the primal subproblems inexactly. We demonstrate that PDHCG maintains a linear convergence rate with an improved convergence constant and is also straightforward to implement on GPUs. Extensive numerical experiments on both synthetic and real-world datasets demonstrate that our method significantly reduces the number of iterations required to achieve the desired accuracy compared to rAPDHG. Additionally, the GPU implementation of our method achieves state-of-the-art performance on large-scale problems. In most large-scale scenarios, our method is approximately 5 times faster than rAPDHG and about 100 times faster than other existing methods. These results highlight the substantial potential of the proposed PDHCG method to greatly improve both the efficiency and scalability of solving complex quadratic programming challenges.

\end{abstract}

\section{Introduction}
In this paper, we consider the general convex quadratic programming (QP) problem, which can be formulated as:
\begin{equation}
    \min_{x \in \mathbb{R}^n} \quad  \frac{1}{2} x^T Q x + c^T x \qquad
\text{s.t.} \quad  A x \leq b, \label{eq:QP}
\end{equation}
where $c\in \mathbb{R}^n$, $b \in \mathbb{R}^m$, $A \in \mathbb{R}^{m \times n}$, and $Q \in \mathbb{R}^{n \times n}$ satisfying $Q \succeq 0$. Problem \eqref{eq:QP} is widely applied across various fields, including engineering, finance, management, and machine learning. In control theory, QP is essential for optimizing control actions to minimize quadratic cost functions, as seen in model predictive control (MPC) \citep{garcia1989model,morari1999model}. In finance, it underpins models such as the Markowitz portfolio selection \citep{markowitz1952portfolio} and the recently developed flow trading model in market design \citep{budish2023flow}. Machine learning applications of QP include support vector machines (SVM) \citep{hearst1998support}, Elastic Net \citep{zou2005regularization}, and quadratic discriminant analysis (QDA) \citep{tharwat2016linear}. Additionally, QP is a fundamental component in many optimization algorithms used in nonlinear programming, such as sequential quadratic programming (SQP) \citep{boggs1995sequential}, which sequentially approximate nonlinear objectives with quadratic functions. 
 
Traditional mainstream solutions for QP include the active-set method \citep{wolfe1959simplex}, and the interior point method (IPM) \citep{vanderbei1999loqo}. These methods deliver reliable, high-accuracy
solutions across various applications and are supported by commercial solvers  like MOSEK \citep{aps2019mosek}, GUROBI \citep{gurobi2021gurobi}, and COPT \citep{ge2022cardinal}, as well as the open-source solvers qpOASES \citep{ferreau2014qpoases} and Clarabel \citep{goulart2024clarabel}. However, their performance diminishes as problem dimensions increase due to the necessity for matrix factorization to solve linear systems. Additionally, matrix factorization, being essentially a sequential algorithm, does not integrate well with modern parallel
computing systems.



With the growing need to tackle large-scale problems and advancements in GPU technology, interest in matrix-free first-order methods (FOM) has increased significantly. SCS \citep{o2016conic}, an Alternating Direction Method of Multipliers (ADMM)-based solver, efficiently handles large-scale convex cone programs, and its latest version \citep{scs} can also address quadratic conic optimization. OSQP \citep{stellato2020osqp}, another ADMM-based solver for large-scale QP problems, achieves performance comparable to commercial solvers like MOSEK and GUROBI, and its GPU implementation, cuOSQP \citep{schubiger2020gpu}, further enhances the algorithm's scalability. Additionally, integrating interior-point methods into first-order frameworks has led to matrix-free solvers like ABIP \citep{lin2021admm,deng2022new}, which combine ADMM's scalability with IPM's stability. ABIP+ \citep{deng2022new} introduces further improvements, including restarts and preconditioning, to boost performance.

Additionally, a significant body of work has focused on the primal-dual hybrid gradient (PDHG) method, which has proven effective in GPU implementations and delivers impressive performance on large-scale problems. The initial application of PDHG in solving linear programs, called PDLP \citep{applegate2021practical}, incorporates a restart scheme to achieve improved numerical outcomes and has been shown to have linear convergence rates \citep{applegate2023faster}. Its GPU adaptations, such as cuPDLP.jl \citep{lu2023cupdlp} and cuPDLP.c \citep{lu2023cupdlpc}, demonstrate significant advantages over Interior-Point Method (IPM)-based and CPU solvers in handling large instances. Building on these developments, \citep{lu2023practical} extends these advancements to QP problems with the introduction of the restarted accelerated PDHG-based QP solver, rAPDHG, which maintains a linear convergence rate and outperforms other QP solvers in various benchmarks.

The contributions of this paper are summarized as follows. First, we introduce a first-order restarted primal-dual hybrid conjugate gradient (PDHCG) algorithm for solving convex QP problems. This algorithm combines the primal-dual hybrid gradient method with the conjugate gradient (CG) method to inexactly solve the primal sub-problem. Specifically, our inexact strategy includes two approaches: (1) PDHCG-adaptive, which employs an adaptive stopping strategy based on a designed accuracy threshold, and (2) PDHCG-fixed, which stops CG after a fixed number of steps. Additionally, we propose several implementation heuristics to further enhance the algorithm’s performance. Extensive numerical results demonstrate that our method significantly reduces the number of iterations
required to achieve the desired accuracy compared to rAPDHG.

Second, we establish the linear convergence of the proposed algorithm by leveraging the quadratic growth property of the smoothed duality gap. Additionally, we theoretically demonstrate that our method significantly reduces the number of iterations required to achieve the desired accuracy compared to rAPDHG. Specifically, as shown in \citep{lu2023practical}, the iteration complexity of rAPDHG depends on both $A$ and $Q$. In contrast, for PDHCG, the dependence of the outer loop iteration complexity on $Q$ is diminished. Although $Q$ still influences the inner-loop convergence of the CG method, the overall dependence on $Q$ is significantly reduced due to the linear convergence property of CG in PDHCG. Our numerical results further support these theoretical findings.

Third, we develop a GPU implementation of the proposed method, which achieves state-of-the-art performance in large-scale scenarios. The GPU implementation is capable of solving QP problems with $10^7$ variables and constraints, and up to $10^{10}$ nonzero elements within 100 seconds. Extensive numerical experiments on both synthetic and real-world datasets demonstrate that the GPU implementation of our method consistently outperforms others on large-scale problems. In most large-scale scenarios, our method is approximately 5 times faster than rAPDHG and about 100 times faster than other existing methods.

The rest of our paper is organized as follows. In \Cref{sec:algor}, we introduce the PDHCG algorithm for convex QP problems, outlining its motivation and advantages over rAPDHG. Then, in \Cref{sec:convergence}, we establish the linear convergence of the proposed algorithm by leveraging the quadratic growth property of the smoothed duality gap. The detailed proofs are provided in Appendices \ref{subsec:proof-adaptive} and \ref{subsec:proof-fixed}. In \Cref{sec:tech}, we introduce several implementation heuristics designed to further enhance the algorithm’s performance. \Cref{sec:exp} presents the empirical evaluation of our method on a variety of QP problems. Lastly, we conclude with closing remarks in \Cref{sec:remark}.

\section{Restarted PDHCG Method}\label{sec:algor}
In this section, we present the restarted primal-dual hybrid conjugate gradient (PDHCG) method for solving the QP problem \eqref{eq:QP}, and outline the motivation behind it as well as its advantages over rAPDHG \citep{lu2023practical}.

\subsection{PDHG-Type Methods}
We first consider the primal-dual formulation of \eqref{eq:QP}
\begin{equation}\label{eq:primal-dual-formulation}
    \min_x \max_{y\geq 0} \; \mathcal{L} (x, y) = \frac{1}{2} x^T
  Q x + c^T x + y^T A x - b^T y.
\end{equation}
The PDHG algorithm \citep{chambolle2011first} (also known as Chambolle-Pock algorithm) is aimed at solving the general convex-concave saddlepoint problem with bilinear term. One-step PDHG iteration is given by:
\[
\begin{aligned}
    x^{k+1}&=\arg \min_x \; \mathcal{L}(x, y^k) +\frac{1}{2\tau_k} \|x-x^k\|_2^2,\\
y^{k+1}&=\arg \max_{y\geq 0} \; \mathcal{L}(2x^{k+1}-x^k, y)- \frac{1}{2\sigma_k} \|y-y^k\|_2^2,
\end{aligned}
\]
where $\tau_k$ and $\sigma_k$ denote the primal and dual step-size. 
Note that the primal update involves solving an unconstrained QP problem, which requires solving a linear system using matrix factorization. In contrast, the dual update only involves a matrix-vector multiplication, making it more computationally efficient and well-suited for execution on GPU or distributed computing architectures.

The PDHG algorithm is known to have only sublinear convergence for general convex-concave saddlepoint problems \citep{chambolle2016ergodic}. However, recent work by \citep{applegate2023faster} demonstrates that a linear convergence rate can be achieved for linear programming (LP) using the sharpness condition and a restart scheme. Additionally, \citep{lu2023practical} applies the restart technique to PDHG for QP problems and also achieves linear convergence rate. Given the demonstrated success of restart scheme in both theory and practice, we incorporate this strategy into our PDHCG algorithm in this paper.

The procedure for the proposed restarted PDHCG algorithm is outlined in Algorithm \ref{algo:Adaptive-CG}. We denote the iterates at the $k$-th inner loop of the 
$n$-th outer loop as $(x^{n,k}, y^{n,k})$, and the running average within the $n$-th outer loop as $(\tilde{x}^{n, k}, \tilde{y}^{n, k})$. The inner loop follows the one-step PDHCG iteration  with primal-dual step size $(\sigma_k, \tau_k)$. The primal variable $x^{n,k}$ is updated using an inexact CG method with specific stopping criteria, which will be detailed later. The dual variable $y^{n,k}$ is updated via a standard PDHG dual step. After $K$ iterations, the inner loop terminates, and the next outer loop restarts using the running averages from the current inner loop.

\begin{algorithm}[t]
\DontPrintSemicolon
\SetAlgoLined
\SetInd{0.5em}{0.5em} 
\SetKwInput{KwInput}{Input}
\SetKwInput{KwOutput}{Output}
\SetKwInput{KwInitialize}{Initialize}
\SetKwFor{Repeat}{Repeat}{}{end} 
\SetKwFor{ForEach}{for each}{}{end} 

\KwInput{initial point $(x^{0, 0}, y^{0, 0})$, step-size $\{(\sigma_k, \tau_k)\}$, extrapolation factor $\theta_k$, restart frequency $K$, average weight $\{\beta_k\}$, sequence of stopping precision $\{\epsilon^k\}_{k=1}^{K}$.}
\KwInitialize{the inner loop. $(\tilde{x}^{n, 0}, \tilde{y}^{n, 0}) \leftarrow (x^{n, 0}, y^{n, 0})$;}
\Repeat{until $(x^{n, 0}, y^{n, 0})$ converges}{
    \ForEach{$k = 0, \ldots, K - 1$}{
        \Indp 
        $x^{k+1} \approx \arg\min _x    \frac{1}{2} x^T Q x + c^T x + (y^k)^T A x - b^T y^k + \frac{1}{2 \tau_k} \| x - x^k \|_2^2$ \;\label{eq:QP-x}
        Use CG method to solve the above problem until the stopping criteria are met\;
        $y^{k+1} = \text{proj}_{\mathbb{R}_+^m} \left\{ y^k - \sigma_k (A (2x^{k+1} - x^k) - b) \right\}$\label{eq:QP-y}\;
        $\tilde{y}^{n, k+1} = \left(1 - \beta_k^{-1}\right)\tilde{y}^{n, k} + \beta_k^{-1} y^{n, k+1}$\;
        $\tilde{x}^{n, k+1} = \left(1 - \beta_k^{-1}\right)\tilde{x}^{n, k} + \beta_k^{-1} x^{n, k+1}$\;
        \Indm 
    }
    Restart the outer loop. $(x^{n + 1, 0}, y^{n + 1, 0}) \leftarrow (\tilde{x}^{n, K}, \tilde{y}^{n, K})$, $n+1 \leftarrow n $.\;
}

\KwOutput{$(x^{n, 0}, y^{n, 0})$}
\caption{Restarted PDHCG}
\label{algo:Adaptive-CG}
\end{algorithm}

\textbf{CG stopping criterion.} Conjugate gradient method \citep{dongarra1991solving} is a matrix-free first-order algorithm for solving linear system that only involves matrix-vector multiplication, which makes it as an ideal method for solving large-scale linear system and unconstrained quadratic minimization problem. We stop CG algorithm either for a fixed number of iterations $N$, or select a series of stopping precision as
\[
    \|x^{n, k} - x_*^{n, k}\| \leq \varepsilon^k,
\]
where $x_*^{n, k}$ is the exact solution to the primal sub-problem at $k$-th iteration,
and $\{\varepsilon^k\}$ is a sequence of accuracy given by \eqref{stop-criteria}. According to \citep{greenbaum1997iterative}, CG method admits the linear convergence property, i.e., let $x^{n,k}_l, l =0,1,2, \dots$ denote the sequence generated by the CG method, then we have 
$$
\|x^{n,k}_l - x^{n,k}_*\| \leq 2 \|x^{n,k}_0 - x^{n,k}_*\| \cdot \left(\frac{\sqrt{\kappa}-1}{\sqrt{\kappa}+1}\right)^l,
$$
where $\kappa$ denote the condition number of the matrix $Q + \frac{1}{\tau_k} I$.  

\subsection{Motivation and Advantages over rAPDHG}
The rAPDHG algorithm \citep{lu2023practical} employs a linearization technique to handle the primal update, approximating the primal objective $\mathcal{L}(x, y^k)$ by performing a linear expansion at an intermediate point $x_{md}^k$. Specifically, rAPDHG uses the following primal update:
\[
x^{k+1} = \arg\min _x  \; \langle x, \nabla \mathcal{L}(x_{md}^k, y^k) \rangle + \frac{1}{2 \tau_k} \| x - x^k \|_2^2.
\]
However, the linearization in rAPDHG becomes highly inaccurate when $Q$ is poorly conditioned, which would significantly increase the iteration complexity of the algorithm. In contrast, the proposed PDHCG algorithm addresses this issue by employing the conjugate gradient method to solve the primal unconstrained QP problem more effectively. The motivation and advantages of our PDHCG method over rAPDHG are twofold:

\begin{table}[tb]
\centering
\caption{Comparison of the theoretical results between rAPDHG and PDHCG}
\begin{threeparttable}
\newcolumntype{C}[1]{>{\centering\arraybackslash}p{#1}}
\begin{tabular}{|c|C{0.45\textwidth}|C{0.25\textwidth}|}
\hline
\thead{Method} & \makecell{outer loop complexity} & \makecell{extra CG steps}\\
\hline
\thead{rAPDHG} &
$\mathcal{O}\left(\big(\|A\|+\sqrt{\|Q\|}+\frac{\|Q\|}{\|A\|}\big)\log\frac{1}{\epsilon}\right)$ & - \\
\hline
\thead{PDHCG-fixed} &
$\mathcal{O}\left(\big(\|A\|+\sqrt{\gamma_K^N \|Q\|}+\frac{\gamma_K^N\|Q\|}{\|A\|}\big)\log\frac{1}{\epsilon}\right)$ & $N$ \\
\hline
\thead{PDHCG-adaptive} &
$\mathcal{O}\left(\|A\|\cdot \log\frac{1}{\epsilon}\right)$ &
$\log_r \frac{\zeta}{  
    2(1 + \tau \|A \|)(1 + \tau \| Q\|) } $\\
\hline
\end{tabular}
\label{tab:theoretical-result}
\begin{tablenotes}
\footnotesize
\item [1] The constants $\gamma_K \in (0, 1)$, $r, \tau, \zeta$ will be specified later.
\end{tablenotes}
\end{threeparttable}
\end{table}

\textbf{Reduced dependency on $Q$.} Compared to rAPDHG \citep{lu2023practical}, our algorithm significantly reduces the dependency on $Q$.  Specifically, we theoretically demonstrate that for PDHCG-fixed (with $N$ CG steps per iteration), the dependence of the outer loop iteration count on $Q$ decreases exponentially with $N$ (see \Cref{thm:fix}). For PDHCG-adaptive, by carefully selecting the stopping criteria precision, the number of outer loop iterations becomes independent of $Q$ (see \Cref{thm:ada}). Although $Q$ still influences the inner-loop convergence of the CG method, the overall dependency on $Q$ is significantly reduced in PDHCG owing to the linear convergence property of CG. Additionally, we provide a theoretical justification for the adaptive CG strategy, showing that the number of iterations required for CG convergence is consistently bounded by a universal constant. This indicates that our PDHCG method closely approximates the efficiency of a single-loop algorithm. A comparison of these theoretical results is summarized in Table \ref{tab:theoretical-result}.



\begin{table}[t]
  \centering
  \caption{The runtime (s) on Random QP with different condition numbers}
\begin{threeparttable}
    \begin{tabular}{ccccccccc}
    \hline
\multirow{2}{*}{Condition} & \multicolumn{3}{c}{PDHCG(GPU)}& \multicolumn{2}{c}{rAPDHG(GPU)} &  SCS(GPU) & OSQP &Gurobi \\
           Number& Iteration & Time & CG extra & Iteration & Time  &Time &Time &Time\\
    \hline
    $10^0$& 406.54  & \textbf{1.74 } & 1375.84  & 795.62  & 2.00   & 4.66  & 439.58  &4574.38  \\
    $10^1$& 586.00  & \textbf{2.08 } & 4092.57  & 1943.55  & 3.11   & 10.36  & 432.65  &4074.02  \\
    $10^2$& 675.51  & \textbf{3.07 } & 7726.04  & 5462.20  & 5.72   & 43.75  & 435.40  &2165.79  \\
    $10^3$& 782.69  & \textbf{3.75 } & 11357.09  & 23041.33  & 17.66   & 127.69  & 432.03  &1743.56  \\
    $10^4$& 1087.15  & \textbf{5.68 } & 15733.61  & 58098.26  & 30.61   & t & 406.51  &1891.62  \\
    $10^5$& 1337.01  & \textbf{7.21 } & 21614.69  & f & f  & t & 396.84  &1606.98  \\
    \hline
    \end{tabular}%
    \label{condition_q}
    \begin{tablenotes}
\footnotesize
\item [1] "f " means the algorithm reaches time limit (7200s), "t" means the algorithm gets wrong solution.
\item [2] "CG extra" means the total number of CG iterations.
\end{tablenotes}
\end{threeparttable}
\end{table}%

To empirically validate our theoretical results, we compare the performance of various algorithms on randomly generated QP problems with varying condition numbers for the matrix $Q$ in the quadratic term. The results are summarized in Table \ref{condition_q}, where the problem size is set as $n = m = 50000$ and sparsity is $10^{-3}$. As shown in Table \ref{condition_q}, the ill-conditioning of $Q$ has a pronounced impact on rAPDHG and SCS, both of which suffer significant degradation in performance. In contrast, our PDHCG method demonstrates much greater resilience to ill-conditioning. A careful examination reveals that the primary difference between rAPDHG and PDHCG lies in the growth of outer loop iterations: for rAPDHG, the iterations increase rapidly, while for PDHCG, they remain within acceptable bounds. Besides, table \ref{condition_q} also shows that, while the condition number of $Q$ affects the inner loop complexity of CG method, the impact is manageable, allowing the overall computational cost to remain robust. Interestingly, both OSQP and Gurobi benefit from ill-conditioned $Q$, as these methods rely on matrix decompositions, where higher condition numbers can, in some cases, simplify the decomposition process.


\textbf{Low-rank acceleration of $Q$.} In various practical applications, such as portfolio optimization, the matrix $Q$ often exhibits a low-rank structure, represented as $Q = P P^T$ where $P \in \mathbb{R}^{n \times k}$ with $k \ll n$. Leveraging this low-rank property allows for significant acceleration of computational tasks, particularly matrix-vector multiplications involving $Qx$, which are required in the primal updates of both the PDHCG and rAPDHG algorithms. The extent of this acceleration, however, varies substantially between the two algorithms. In rAPDHG, the matrix-vector multiplication $Qx$ is performed once per iteration. In contrast, PDHCG requires $Qx$ to be computed within each inner iteration of the CG method, which constitutes the primary computational bottleneck of PDHCG. Consequently, the benefits of accelerating the computation of $Qx$ are more substantial for PDHCG.

To illustrate this difference, we conducted experiments where $Q = P P^T$ with increasing rank $k$, and we compared the performance of PDHCG and rAPDHG with and without the low-rank acceleration technique. As demonstrated in Table \ref{tab:PP-gpu} and Table \ref{tab:PP-cpu}, the speedup achieved by exploiting the low-rank structure of $Q$ is significantly more pronounced in the PDHCG algorithm than in rAPDHG, providing a distinct computational advantage for PDHCG in solving such problems.

\begin{table}[htb]
    \centering
    \begin{minipage}[t]{0.45\textwidth}
        \centering
        \caption{10 Random QP problems in GPU ,n = 1e7,sparsity = 1e-4, 'o' means out of memory.}
        \label{tab:PP-gpu}
        \resizebox{\textwidth}{!}{
        \begin{tabular}{c|ccc|ccc} 
		\hline
		\multirow{2}[4]{*}{rank} & \multicolumn{3}{c}{PDHCG} & \multicolumn{3}{c}{rAPDHG} \\
		\cmidrule{2-7}          & P'Px  & Qx    &Ratio  & P'Px  & Qx    &Ratio  \\
		\hline
		5000  & \bf{3.45}  & 5.74  & \bf{0.60}  & 13.96  & 19.65  & 0.71   \\
    10000 & \bf{3.70}  & 9.21  & \bf{0.40} & 14.64  & 19.82  & 0.74   \\
    20000 & \bf{4.31}  & 17.27  & \bf{0.25}& 15.53  & 28.81  & 0.54   \\
    50000 & \bf{5.52}  & 48.90  & \bf{0.11}& 16.21  & 43.54  & 0.37   \\
    100000 & \bf{7.10}& 105.71  & \bf{0.07}& 19.36  & 81.76  & 0.24   \\
    200000 & \bf{11.93}  & o& -     & 20.56  & o& -  \\
		\hline
	\end{tabular}}
    
    \end{minipage}%
    \hspace{0.05\textwidth}
    \begin{minipage}[t]{0.45\textwidth}
        \centering
        \caption{10 Random QP problems in CPU, n = 2e5,sparsity = 1e-3.}
        \label{tab:PP-cpu}
        \resizebox{\textwidth}{!}{
        \begin{tabular}{c|ccc|ccc} 
    \hline
    \multirow{2}[4]{*}{rank} & \multicolumn{3}{c}{PDHCG} & \multicolumn{3}{c}{rAPDHG} \\
\cmidrule{2-7}          & P'Px  & Qx    & Ratio & P'Px  & Qx    & Ratio \\
    \hline
    10    & 44.38  & \bf{27.33}  & 1.62  & 61.31  & 64.58  & \bf{0.95}  \\
    50    & \bf{24.77}  & 44.62  & \bf{0.56}  & 67.87  & 78.51  & 0.86  \\
    100   & \bf{36.89}  & 73.37  & \bf{0.50}  & 74.78  & 109.28  & 0.68  \\
    200   & \bf{19.32}  & 134.03  & \bf{0.14}  & 86.87  & 175.07  & 0.50  \\
    1000  & \bf{48.41}  & 504.33  & \bf{0.10}  & 94.69  & 524.76  & 0.18  \\
    2000  & \bf{33.99}  & 1351.99  & \bf{0.03}  & 84.28  & 767.83  & 0.11  \\
    \hline
    \end{tabular}}
    \end{minipage}
\end{table}
\section{Convergence Analysis} \label{sec:convergence}
In this section, we present theoretical guarantees for the linear convergence of \Cref{algo:Adaptive-CG} with both fixed and adaptive CG stopping strategies. 
To begin, we rewrite the minimax formulation \eqref{eq:primal-dual-formulation} for QP problems as a general saddle-point problem:
\begin{equation}\label{prob:general-form}
  \min_x \max_y G (x) +\langle A x, y \rangle - F (y) 
\end{equation}
where $G (x) = \frac{1}{2} x^T Q x + c^T x$ and $F (y) = b^T y
+\mathbb{I} (y \geq 0)$ are convex functions. 

Before presenting the formal convergence results, we introduce two progress metrics. The first one is the duality gap, a metric commonly used in the analysis of primal-dual algorithms \citep{chambolle2011first}.  Let $z = (x, y) \in \mathbb{R}^{n+m}$ and define the feasible set $\mathcal{Z} = \mathbb{R}^{n} \times \mathbb{R}^m_+$. Besides, denote by $\mathcal{Z}^{\ast}$ the set of optimal solutions to \eqref{prob:general-form}.

\begin{definition}[Duality Gap]
    For any $z, \hat{z} \in \mathcal{Z}$ with $z=(x,y)$ and $\hat{z}=(\hat{x}, \hat{y})$, let
  \[ 
  Q (z, \hat{z}) = \mathcal{L} (x, \hat{y}) - \mathcal{L} (\hat{x}, y), 
  \]
then $\max_{\hat{z} \in \mathcal{Z} } Q (z, \hat{z})$ represents the duality gap at $z$.
\end{definition}

The potential for an unbounded feasible region in QP problems may result in an infinite duality gap. To mitigate this issue and obtain more robust convergence guarantees, we adopt a novel progress metric introduced by \citep{fercoq2022quadratic}. This metric modifies the traditional duality gap by incorporating a smoothness penalty, ensuring the gap remains finite. The smoothed duality gap is crucial in proving the linear convergence of general restarted PDHG algorithms and has also been utilized in the analysis of rAPDHG \citep{lu2023practical}.

\begin{definition}[Smoothed Duality Gap \citep{fercoq2022quadratic}]
  For any $\xi \geq 0$ and $z, \dot{z} \in \mathcal{Z}$, the smoothed duality gap at $z$ centered at $\dot{z}$ is defined as
  \[ G_{\xi} (z, \dot{z}) \assign \max_{\hat{z} \in \mathcal{Z}} \{ Q (z,
     \hat{z}) - \xi \| \hat{z} - \dot{z} \|^2 \}. \]
\end{definition}
Moreover, \citep{fercoq2022quadratic} employs the smoothed duality gap to define a new regularity condition, which ensures that the smoothed duality gap exhibits quadratic growth with respect to the distance from the optimal solution. Problems satisfying this condition can be shown to achieve linear convergence. Additionally, \citep{lu2023practical} has validated that convex QP problems meet this quadratic growth condition, thereby guaranteeing linear convergence for these cases.
\begin{definition}[Quadratic Growth]
   Let $\xi > 0$ and $\mathbb{S}$ be any bounded set. We say that the smoothed duality gap satisfies quadratic growth on set $\mathbb{S}$ if there exists some $\alpha_{\xi} > 0$ such that, for all $z^{\ast} \in \mathcal{Z}^{\ast}$ and any $z \in \mathbb{S}$, 
  \[ G_{\xi} (z, z^{\ast}) \geq \alpha_{\xi} \tmop{dist}^2 (z,
     \mathcal{Z}^{\ast}). \]
\end{definition}

\vspace{3mm}
\subsection{PDHCG with Fixed Stopping Strategy}
For the PDHCG-fixed algorithm, where the maximum number of CG iterations is fixed at $N$, we show that, under mild assumptions on the lower bound of $N$, the outer-loop maintains a linear convergence rate.  Moreover, by leveraging the linear convergence rate of the CG method, we derive convergence results that exhibit an exponentially decaying dependence on $Q$.

\begin{theorem}[Convergence Result of PDHCG-fixed] \label{thm:fix}
  Let $\{ z^{n, 0} \}^{\infty}_{n = 0}$ denote the sequence generated by PDHCG-fixed with a fixed number of CG iterations $N \geq \log_{\gamma_K} \frac{1}{K}$ where $\gamma_K$ represents the convergence rate of CG in the $K$-th inner loop. Suppose for any $\xi > 0$, the function $\mathcal{L}$ exhibits quadratic growth with parameter $\alpha_{\xi}$ within a ball $B_R (z^{0, 0})$ centered at $z^{0, 0}$, where the radius $R = \frac{2}{1 - 1 / e} \tmop{dist} (z^{0, 0}, \mathcal{Z}^{\ast})$. Choose
  \[ \theta_k = \frac{k}{k+1}, \beta_k = k, \tau_k = \frac{k + 1}{2 (\gamma^N_K \| Q \| + K \| A \|)}, \sigma_k = \frac{k+1}{2 K\| A \|}, \]
  and ensure the restart frequency $K$ satisfies
  \[ K \geq \max \left\{ \frac{4 \| A \|}{\xi}, \sqrt{\frac{4 \gamma^N_K \| Q \|}{\xi}}, \frac{4 e^2 \| A \|}{\alpha_{\xi}}, \sqrt{\frac{4 e^2 \gamma^N_K \| Q \|}{\alpha_{\xi}}}, \frac{\gamma^N_K \| Q \|}{\| A \|} \right\}. \]
  Then, for any outer iteration $n$, the following holds:
  \begin{itemize}
      \item[(1)] The sequence $\{ z^{n, 0} \}$ is bounded and remains within $B_R (z^{0, 0})$:
  \[ \| z^{n, 0} - z^{0, 0} \| \leq \frac{2}{1 - 1 / e} \tmop{dist} (z^{0, 0}, \mathcal{Z}^{\ast}). \]
  \item[(2)] The sequence $\{ z^{n, 0} \}$ converges linearly to an optimal solution:
  \[ \tmop{dist} (z^{n, 0}, \mathcal{Z}^{\ast}) \leq e^{- n} \tmop{dist} (z^{0, 0}, \mathcal{Z}^{\ast}). \]
  \end{itemize}
\end{theorem}
The convergence result of PDHCG-fixed confirms that, under specific conditions, the algorithm consistently converges to an optimal solution. This reliability is achieved through careful parameter selection and adequate restart frequency, which ensure iterate boundedness and a linear convergence rate. Such guarantees affirm the algorithm’s efficacy in solving optimization problems characterized by quadratic growth properties. The detailed proof of Theorem \ref{thm:fix} is provided in Appendix \ref{subsec:proof-fixed}.

\subsection{PDHCG with Adaptive Stopping Strategy}
For each inner loop, we fix the primal and dual step size $\sigma = \tau = \frac{1}{2 \| A \|}$, and define the sequence of precision of CG as 
 \begin{equation}\label{stop-criteria}
        \varepsilon^1  = \frac{\zeta \| y^{n,1} - y^{n,0} \|}{1 + \tau \| Q\|},\quad 
    \varepsilon^k  = \varepsilon^{k-1} + \frac{\zeta \| z^{n, k-1} - z^{n, k-2} \|}{1 + \tau \| Q\|}\quad \text{for}\quad  2 \leq k \leq K,
\end{equation}
where $\zeta$ is some positive constant satisfying $2 \zeta \leq  \min \left\{ \frac{1 - \sqrt{\sigma \tau} \| A \|}{2 \sqrt{\sigma
     \tau} K^2}, \frac{1 - \sqrt{\sigma \tau} \| A \|}{2 \tau K^2} \right\}.$
Intuitively, as the iteration sequence approaches the optimum, we progressively relax the stopping conditions, resulting in fewer CG iterations being required in the later stages. The following theorem asserts that the PDHCG-adaptive algorithm, with these adaptive stopping criteria, guarantees linear convergence. Notably, this approach removes the dependence on $\|Q\|$, which is present in standard PDHG, thus improving scalability and efficiency.

\begin{theorem}[Convergence Result of PDHCG-adaptive] \label{thm:ada}
  Let $\{ z^{n, 0} \}^{\infty}_{n = 0}$ denote the sequence generated by PDHCG-adaptive with target CG precision as defined in \eqref{stop-criteria}. Suppose for any $\xi > 0$, the function $\mathcal{L}$ exhibits quadratic growth with parameter $\alpha_{\xi}$ within a ball $B_R (z^{0, 0})$ centered at $z^{0, 0}$, where the radius $R = \frac{3}{1 - 1 / e} \tmop{dist} (z^{0, 0}, \mathcal{Z}^{\ast})$. Choose $\theta_k = 1$, $\beta_k = k$, primal and dual step sizes as $\sigma = \tau = \frac{1}{2 \| A \|}$, and ensure that the restart frequency $K$ satisfies
  \[ K \geq \max \left\{ \frac{4 \| A \|}{\xi}, \frac{2 e^2 \| A
     \|}{\alpha_{\xi}} \right\}. \]
  Then, for any outer iteration $n$, the following holds:
  \begin{itemize}
      \item[(1)] The sequence $\{ z^{n, 0} \}$ is bounded and remains within $B_R (z^{0, 0})$:
  \[ \| z^{n, 0} - z^{0, 0} \| \leq \frac{3}{1 - 1 / e} \tmop{dist}
     (z^{0, 0}, \mathcal{Z}^{\ast}). \]
    \item[(2)] The sequence $\{ z^{n, 0} \}$ converges linearly to an optimal solution:
  \[ \tmop{dist} (z^{n, 0}, \mathcal{Z}^{\ast}) \leq e^{- n} \tmop{dist}
     (z^{0, 0}, \mathcal{Z}^{\ast}). \]
  \end{itemize}
\end{theorem}
Furthermore, we can prove that the number of CG iterations is upper-bounded by a dimension-free constant when the stopping precision is selected according to \eqref{stop-criteria}. This ensures that the computational overhead associated with CG iterations remains manageable. The detailed proofs of Theorem \ref{thm:ada} and Proposition \ref{cor:ada} are provided in Appendix \ref{subsec:proof-adaptive}.

\begin{proposition}[Upper Bound for Adaptive-CG Iteration] \label{cor:ada}
Under the choice of stopping criteria \eqref{stop-criteria}, the total number of iterations required in the inexact CG step is bounded by the constant $\log_r \frac{\zeta}{2  
    (1 + \tau \|A \|)(1 + \tau \| Q\|) } $, where $r = \frac{\sqrt{\kappa} - 1}{\sqrt{\kappa} + 1}$ and $\kappa$ represents the condition number of $Q + \frac{1}{2 \tau} I$.
\end{proposition}

\section{Algorithm Enhancement Process}\label{sec:tech}
In this section, we introduce several heuristic techniques aimed at enhancing the performance of our algorithm. These techniques are designed to improve both the convergence speed and the computational efficiency of the algorithm, particularly in large-scale problem instances. \\
\textbf{Progress metric.} We utilize the relative KKT error as the progress metric to compare the performance of different algorithms. This metric is consistent with those used in \citep{applegate2021practical} and \citep{lu2023practical}. For a solution \( z = (x, y) \in \mathcal{Z}\), the relative primal residual, dual residual, and primal-dual gap for the problem in \eqref{eq:primal-dual-formulation} are defined as follows:
\[
r_{\text{primal}} = \frac{\| [Ax - b]^+ \|_\infty}{1 + \max\{\|Ax\|_\infty, \|b\|_\infty\}},
\]
\[
r_{\text{dual}} = \frac{\|Qx + A^T y + c\|_\infty}{1 + \max\{\|Qx\|_\infty, \|A^T y\|_\infty, \|c\|_\infty\}},
\]
\[
r_{\text{gap}} = \frac{|x^T Qx + c^T x + b^T y|}{1 + \max\left\{\left| \frac{1}{2} x^T Qx + c^T x \right|, \left| \frac{1}{2} x^T Qx + b^T y \right|\right\}}.
\]
The relative KKT error is defined as the maximum of these three terms:
\[
\text{relKKT}(z) = \max\{r_{\text{primal}}, r_{\text{dual}}, r_{\text{gap}}\}.
\]
The algorithm terminates when the relative KKT error is smaller than the termination tolerance \(\epsilon\), i.e., \(
\text{relKKT}(z) \leq \epsilon. 
\) We consider two relative KKT error levels: \(\epsilon = 10^{-3}\) as low accuracy and \(\epsilon = 10^{-6}\) as high accuracy.\\
\textbf{Augmented penalty method.} Let \( Gx = q \) represent the equality constraints within the system \( Ax \leq b \). To further enhance the iteration process, we incorporate the augmented penalty term \(\|Gx - q\|_2^2\) into the objective function. Then, the minimax formulation becomes:
\[
    \min_x \max_{y \geq 0}  \mathcal{L} (x, y) \assign \frac{1}{2} x^T
  Q x + c^T x + \frac{\rho}{2}\|Gx - q\|_2^2 + y^T A x - b^T y 
\]
where the parameter $\rho$ is chosen as \(\rho =\frac{0.1\|Q\| _2}{\|G^T G \| _2} \). It is important to note that incorporating the augmented penalty term can increase the sparsity of the problem. Therefore, for large-scale problems with significant sparsity, we set $\rho = 0$, as the penalty term might otherwise hinder performance by introducing unnecessary complexity.\\
\textbf{Preprocessing.} First-order methods (FOMs) may suffer from slow convergence when solving real-world instances due to their ill-conditioned nature. To mitigate this, we use the diagonal preconditioning heuristic developed in \citep{applegate2021practical} for linear programming to rescale the problems and improve the condition number. Specifically, we rescale the matrix \(A\) to \(\tilde{A} = D_1 A D_2\) with positive definite diagonal matrices \(D_1\) and \(D_2\). The vectors \(b\) and \(c\) are correspondingly rescaled to \(\tilde{b} = D_1 b\) and \(\tilde{c} = D_2 c\). Additionally, we modify \( Q \) to \( Q + \rho G^T G \) and  rescale it to \(D_2 (Q + \rho G^T G) D_2\). The matrices \( D_1 \) and \( D_2 \) are obtained by running 10 steps of Ruiz scaling followed by a Pock-Chambolle scaling on the matrix \(\begin{pmatrix} Q + \rho G^T G & A^T \\ A & 0 \end{pmatrix}\). \\
\textbf{Switching between Barzilai-Borwein (BB) and CG algorithms.} In our implementation, we first check whether box constraints are present. If no box constraints are detected, we utilize the CG algorithm to solve the primal sub-problem. On the other hand, if box constraints are present, the Barzilai-Borwein (BB) algorithm is selected, as described in \citep{dai2005projected}. The BB algorithm is a projected gradient method that employs a specific step-size to ensure efficient convergence, particularly in constrained optimization scenarios. Specifically, the updating rule of BB algorithm is given by:
\[
x^{n,k+1}=\text{proj}_{x \in \mathcal{X}} \left(x^{n,k}-\frac{1}{\alpha_{n,k}}g_{n,k} \right), \quad \text{with} \quad \alpha_{n,k}=\frac{s^T _{n,k-1}y_{n,k-1}}{s^T _{n,k-1}s_{n,k-1}} \]
where \(g_{n,k}\) is the gradient at \(x_{n,k}\), \(s_{n,k-1}=x^{n,k}-x^{n,k-1}\) and \(y_{n,k-1}=g_{n,k}-g_{n,k-1}\). During each inner loop, the stopping criterion for CG is set as \( \|x^{n,k} - x^{n,k-1}\| \leq \epsilon_k \) where
\[
\epsilon_0 = 0 \quad \text{and} \quad \epsilon_k = \epsilon_{k-1} + 0.05 \text{relKKT}(z^{n,k -1}).
\]
\textbf{Adaptive restart.} We primarily adopt the same restarting strategy as described in Section 6 of \citep{lu2023practical}, with the parameter settings adjusted to suit our specific problem:  
\[\beta_{\text{sufficient}} = 0.2, \, \beta_{\text{necessary}} = 0.8, \text{ and } \beta_{\text{artificial}} = 0.2.
\]
\textbf{Adaptive step-size.} 
Our adaptive step-size strategy is inspired by the approach used in PDLP \citep{applegate2021practical}, with a novel step-size limitation specifically tailored for QP problems:
\[
\eta \leq \frac{\| z^{n,k + 1} - z^{n,k} \|^2_w}{2(x^{n,k + 1} - x^{n,k})^T A^T (y^{n,k + 1} - y^{n,k}) + (x^{n,k + 1} - x^{n,k})^T (Q + \rho G^T G) (x^{n,k + 1} - x^{n,k}) }.
\]
\textbf{Primal weight adjustment.} For the adaptive primal weight, we update the primal weight only at the beginning of each new epoch, using a weight factor of \(\theta = 0.2\). The update follows the rule:
\[
\texttt{PWeight}(z^{n,0}, z^{n-1,0}, \omega^{n-1}) :=
\begin{cases}
\exp \left( \theta \log \left( \frac{\Delta y^n}{\Delta x^n} \right) + (1 - \theta) \omega^{n-1} \right), & \Delta x^n, \Delta y^n > \epsilon_{\text{zero}} \\
\omega^{n-1}, & \text{otherwise,}
\end{cases}
\]
where \( \Delta x^n = \|x^{n,0} - x^{n-1,0}\|_2 \) and \( \Delta y^n = \|y^{n,0} - y^{n-1,0}\|_2 \).
  
\section{Numerical Experiments}\label{sec:exp}
In this section, we introduce PDHCG.jl, a prototype QP solver based on the PDHCG algorithm, available on GitHub\footnote{The code is available on \href{https://github.com/Huangyc98/PDHCG.jl.git}{https://github.com/Huangyc98/PDHCG.jl}}. It is important to note that PDHCG.jl is a preliminary implementation in Julia, primarily aimed at demonstrating the potential of this method. We compare the numerical performance of PDHCG against several popular QP solvers: rAPDHG \citep{lu2023practical}, SCS \citep{o2016conic}, OSQP \citep{stellato2020osqp}, and COPT \citep{ge2022cardinal}.

This section is organized into three parts. In \Cref{Large-scale synthetic QP}, we generate seven types of synthetic QP problems with dimensions ranging from \(10^3\) to \(10^7\) to evaluate the scalability and performance of our algorithm on large-scale problems. In \Cref{realworldqp}, we test the solvers on large-scale, real-world datasets from LIBSVM \citep{chang2011libsvm} and the SuiteSparse Matrix Collection \citep{davis2011university}, transforming these datasets into standard Lasso problems for a comparative performance analysis. Finally, in \Cref{twobenchmark}, we assess the stability of our algorithm on small to medium-sized problems by evaluating it on standard convex QP datasets from the Maros–Mészáros benchmark \citep{Maros1999qp} and QPLIB relaxations \citep{furini2019qplib}.

\textbf{Computing environment.} 
We conduct our experiments on a high-performance computing cluster. Specifically, we utilize an NVIDIA H100 GPU with 80GB HBM3 memory, running CUDA version 12.4, for GPU-based solvers. The CPU used is an Intel Xeon Platinum 8469C operating at 2.60 GHz, equipped with 512 GB of RAM. For all experiments, we allocate 32 CPU cores and the experiments are executed using Julia version 1.10.3.

\textbf{Solvers.} 
PDHCG.jl is implemented as an open-source Julia module, utilizing CUDA.jl as the interface for running on NVIDIA CUDA GPUs within Julia. We compare PDHCG.jl against both the CPU and GPU versions of rAPDHG \citep{lu2023practical}, the indirect and direct GPU and CPU versions of SCS \citep{o2016conic}, the CPU version of OSQP \citep{stellato2020osqp}, and the commercial solver COPT \citep{ge2022cardinal}. It is worth noting that rAPDHG, SCS, and OSQP are first-order QP solvers, while COPT is a second-order solver. Besides, SCS (Direct) and OSQP require matrix factorization as part of their solution process. 

\subsection{Large-scale Synthetic QP}
\label{Large-scale synthetic QP}
In this section, we present numerical results for seven artificially generated problem sets: Random QP, Equality constrained QP, Portfolio optimization, Optimal control problem, Lasso, SVM, and Huber regression. These problem types are widely recognized and commonly applied in real-world scenarios. To evaluate and compare the scalability and performance of algorithms on large-scale problems, we generated these types of QP problems with dimensions ranging from $10^3$ to $10^7$. The generation of matrices and coefficients for these problems, unless otherwise specified, follows the same procedure as described in OSQP \citep{stellato2020osqp}. 


\vspace{2mm}
\textbf{Random QPs.}
We first consider the general QP model:
\[\begin{array}{cc}
  \min & \frac{1}{2}x^T Q x + c^T x\\
  s.t. & {lb} \leqslant {Ax} \leqslant {ub},
\end{array}\]
where the problem instances have \( n \) variables and \( n \) constraints. To ensure the quadratic term is positive semidefinite, we generate the matrix \( Q = PP^T + \alpha I \), where \( P \in \mathbb{R}^{ n \times 1000} \) with sparsity $10^{-4}$ and regularization term \(\alpha = 10^{-2}\). For problems with equality constraints, we simply set \( ub = lb \).

\begin{table}[!ht]
	\centering
 \caption{Solving time (s) over 10 Random QP problems}
 \begin{threeparttable}
	\begin{tabular}{c|cc|cc|ccc|c|c} 
        \hline
	\multirow{2}[4]{*}{n} & \multicolumn{2}{c|}{PDHCG} &        
        \multicolumn{2}{c|}{rAPDHG} & \multicolumn{3}{c|}{SCS} & OSQP  & GUROBI \\
		\cmidrule{2-10}          & GPU   & CPU   & GPU   & CPU   & GPU   & ID CPU & D CPU & CPU   & CPU \\
		\hline
		1E+03 & 1.08  & 0.69  & 1.51  & 0.69  & 0.29  & 0.30  & 0.00  & 0.01  & \textbf{0.00} \\
		1E+04 & 2.13  & 1.60  & 2.74  & 1.30  & 1.66  & 2.64  & 0.19  & 0.08  & 0.11  \\
		1E+05 & \textbf{4.89} & 30.82  & 6.97  & 36.14  & 10.44  & 81.89  & f     & f     & 360.00  \\
		1E+06 & \textbf{5.44} & 817.82  & 20.46  & 3494.74  & 1957.14  & f     & f     & f     & f \\
		1E+07 & \textbf{19.57} & f     & 124.25  & f     & f     & f     & f     & f     & f \\
		\hline
	\end{tabular}%
 \label{randomqp_1}
 \begin{tablenotes}
\footnotesize
\item [1] "f " means the algorithm reaches time limit (3600s).
\end{tablenotes}
\end{threeparttable}
\end{table}%

\vspace{-8mm}
\begin{table}[!ht]
  \centering
  \caption{Solving time (s) over 10 Random QP problems with equality constraints}
    \begin{tabular}{c|cc|cc|ccc|c|c}
    \hline
    \multirow{2}[4]{*}{n} & \multicolumn{2}{c|}{PDHCG} & \multicolumn{2}{c|}{rAPDHG} & \multicolumn{3}{c|}{SCS} & OSQP  & COPT \\
    \cmidrule{2-10}          & GPU   & CPU   & GPU   & CPU   & \multicolumn{1}{c}{GPU} & ID CPU & D CPU & CPU   & CPU \\
    \hline
    1E+03 & 0.72  & 0.50  & 0.75  & 0.55  & \multicolumn{1}{c}{0.04 } & 0.01  & 0.06  & \textbf{0.01 } & 0.02 \\
    1E+04 & 1.62  & 0.96  & 1.53  & 1.14  & \multicolumn{1}{c}{3.97 } & 1.84  & 0.13  & \textbf{0.01 } & 0.02 \\
    1E+05 & \textbf{1.73} & 24.37  & 1.84  & 19.28  & \multicolumn{1}{c}{7.39 } & 3.56  & f     & f     & 59.93 \\
    1E+06 & \textbf{3.02} & 1076.39  & 13.55  & f     & \multicolumn{1}{c}{1611.21 } & 560.00  & f     & f     & f \\
    1E+07 & \textbf{26.97} & f     & 191.72  & f     & \multicolumn{1}{c}{f}    & f     & f     & f     & f \\
    \hline
    \end{tabular}%
  \label{randomqp_eq}%
\end{table}%

Tables \ref{randomqp_1} and \ref{randomqp_eq} present the geometric mean of solving times (in seconds) for 10 randomly generated QP problems and 10 randomly generated equality-constrained QP problems, respectively. For small-scale problems (with dimensions \(n = 10^3, 10^4\)), solvers such as COPT, OSQP, and SCS (Direct) outperform others due to the lower computational costs associated with low-dimensional matrix decomposition. However, as the problem dimension exceeds \(10^5\),  these solvers struggle to solve the problems due to the high computational cost of matrix decomposition. In contrast, the computational advantage of GPU-based and matrix-free algorithms becomes evident as the problem dimensions surpass \(10^5\). For large-scale problems with dimensions \(n = 10^6\) and \(10^7\), the GPU version of PDHCG is approximately 5 times faster than the GPU version of rAPDHG and over 100 times faster than the other solvers.

Table \ref{sparsity} presents the results for 10 randomly generated QP problems with 
$n=10^4$ and varying levels of sparsity in both the constraint matrix $A$ and the objective matrix $Q$. As the sparsity decreases (i.e., the problems become denser), the solving time for all solvers increases. However, PDHCG exhibits the smallest increase in solving time, indicating that it is more efficient and scales better on denser problems compared to the other solvers.

\begin{table}[!ht]
	\centering
 \caption{Solving time (s) on Random QP with different sparsity of $A$ and $Q$}
	\begin{tabular}{c|c|c|ccc}
		\hline
		Sparsity & PDHCG(GPU) & rAPDHG(GPU) & SCS(GPU) & OSQP  & GUROBI \\
		\hline
		1E-04 & 2.41  & 2.00  & 3.17  & 0.22  & \textbf{0.03 } \\
		1E-03 & \textbf{2.02} & 2.43  & 6.14  & 413.64  & 9.62  \\
		1E-02  & \textbf{2.54} & 4.63  & 20.69  & 682.29  & 160.06  \\
		1E-01   & \textbf{5.27} & 28.49  & 37.66  & 1115.06  & 240.73  \\
		1E+00     & \textbf{8.39} & 128.51  & 52.56  & f     & 623.17  \\
		\hline
	\end{tabular}
 \label{sparsity}%
\end{table}

\vspace{2mm}

\textbf{Portfolio Optimization.} Portfolio selection is a well-known problem in finance that aims to allocate assets in a way that maximizes the risk-adjusted return. The standard formulation is
\[
\begin{aligned}
& \text{maximize} \quad \mu^T x - \gamma \left( x^T \Sigma x \right) \\
& \text{subject to} \quad 1^T x = 1, \, x \geq 0,
\end{aligned}
\]
where $\mu$  represents the expected returns, $\Sigma$ is the covariance matrix of asset returns, and $\gamma$ is the risk aversion parameter. The risk model is typically expressed as $\Sigma = F F^T + D$ where $F \in \mathbb{R}^{n \times k}$ captures factor loadings and $D$ is a diagonal matrix of specific risks. To simplify the problem, we introduce a new variable \( y = F^T x \) and reformulate the problem as
\[
\begin{aligned}
& \text{minimize} \quad x^T D x + y^T y - \gamma^{-1} \mu^T x \\
& \text{subject to} \quad y = F^T x, \, 1^T x = 1, \, x \geq 0.
\end{aligned}
\]
In the experiment, we set the risk aversion parameter $\gamma = 1$ and generated portfolio problems with the number of assets $n$ ranging from $10^3$ to $10^7$ and the number of factors $k = n$.

\begin{table}[ht]
	\centering
 \caption{Solving time(s) on 10 Portfolio problems} 
	\begin{tabular}{c|cc|cc|ccc|c|c}
		\hline
		\multirow{2}[2]{*}{n} & \multicolumn{2}{c|}{PDHCG} & \multicolumn{2}{c|}{rAPDHG} & \multicolumn{3}{c|}{SCS} & OSQP   & COPT \\
		\cmidrule{2-10} & GPU   & CPU   & GPU   & CPU   & GPU & ID CPU & D CPU & CPU   & CPU \\
		\hline
	1E+03  & 0.64  & 0.53  & 0.73  & 0.54  & 0.26  & 0.25  & 0.17  & 0.64  & \textbf{0.04 } \\
    1E+04 & 1.37  & 1.63  & 2.10  & 4.35  & 5.71  & 7.26  & 0.67  & 17.61  & \textbf{0.21 } \\
    1E+05 & \textbf{1.70 } & 33.55  & 3.60  & 58.30  & 172.42  & 113.59  & f     & 402.85  & 2364.76  \\
    1E+06 & \textbf{8.25 } & 2309.99  & 39.58  & f     & f     & f     & f     & f     & f \\
    1E+07 & \textbf{92.29 } & f     & 346.38  & f     & f     & f     & f     & f     & f \\
		\hline
	\end{tabular}%
	\label{Portfolio}%
\end{table}%

Table \ref{Portfolio} summarizes the geometric mean of solving times for different solvers on 10 randomly generated portfolio optimization problems. The results are consistent with those observed for the Random QP problems: solvers utilizing matrix decomposition, such as OSQP and COPT, perform well on small-scale problems. However, for large-scale cases, the GPU version of PDHCG significantly outperforms the others, being approximately 5 times faster than the GPU version of rAPDHG and more than 100 times faster than the other solvers.

\vspace{2mm}

\textbf{Model Predictive Control.}
The problem of controlling a constrained linear time-invariant dynamical system can be formulated as an optimization problem. The objective is to minimize a cost function that accounts for both the system's state and control effort over a finite time horizon:
\[
\begin{aligned}
    & \text{minimize} \quad x_T^T Q_T x_T + \sum_{t=0}^{T-1} \left( x_t^T Q x_t + u_t^T R u_t \right) \\
    & \text{subject to} \quad x_{t+1} = A x_t + B u_t \\
    & \quad \quad \quad \quad \quad x_t \in \mathcal{X}, \; u_t \in \mathcal{U} \\
    & \quad \quad \quad \quad \quad x_0 = x_{\text{init}},
\end{aligned}
\]
where $x_t$ and $u_t$ represent the state and control variables, respectively, at time $t$. 
In the experiment, to ensure \( Q \) and \( Q_T \) are positive semi-definite, we generate \( Q \) using the formula \( Q = P P^T + \alpha I \), and set \( Q_T = Q \). The matrices \( A \), \( B \), and \( R \) are generated following the procedure outlined in OSQP, and we vary the time horizon \( T \) to control the problem's dimensionality.

\begin{table}[!ht]
  \centering
  \caption{Solving time (s) over 10 MPC problems}
  \begin{threeparttable}
    \begin{tabular}{c|cc|cc|ccc|c|c}
    \toprule
    \multirow{2}[4]{*}{n} & \multicolumn{2}{c|}{PDHCG} & \multicolumn{2}{c|}{rAPDHG} & \multicolumn{3}{c|}{SCS} & OSQP  & COPT \\
\cmidrule{2-10}          & GPU   & CPU   & GPU   & CPU   & GPU   & ID CPU & D CPU & CPU   & CPU \\
    \midrule
    1E+03 & 2.55  & 0.82  & 3.14  & 0.96  & t     & 18.93  & 0.06  & \textbf{0.00 } & 0.13  \\
    1.E+04 & 4.02  & 7.98  & 5.50  & 9.36  & t     & 21.41  & 0.16  & \textbf{0.03 } & 0.78  \\
    1E+05 & 4.91  & 43.83  & 8.57  & 89.41  & t     & 22.42  & 1.41  & \textbf{0.39 } & 8.18  \\
    1E+06 & \textbf{7.52 } & 782.43  & 12.23  & 1360.88  & t     & 52.10  & 17.27  & 7.95  & 90.00  \\
    1E+07 & \textbf{37.97 } & f     & 68.09  & f     & t     & 322.91 & 138.56  & 96.72  & 1210.50  \\
    \bottomrule
    \end{tabular}%
  \label{mpc}%
   \begin{tablenotes}
\footnotesize
\item [1] "t " means numerical error in solving process .
\end{tablenotes}
\end{threeparttable}
\end{table}%



Table \ref{mpc} presents the geometric mean of solving times across 10 model predictive control (MPC) problems evaluated using various solvers, with problem dimensions ranging from \(10^3\) to \(10^7\). Unlike previous synthetic problems, MPC instances are structured such that both the objective matrix and the constraint matrix are block diagonal. This structure allows solvers like OSQP, COPT, and SCS to handle problems up to \(10^7\) dimensions.  While OSQP performs well for smaller problem sizes (\(10^3\) and \(10^4\)), the GPU-based PDHCG solver dominates for larger problems, outperforming rAPDHG by 1.5x and delivering at least a twofold improvement over other solvers in the \(10^6\) and \(10^7\) range.

\vspace{2mm}

\textbf{Lasso, SVM snd Huber Regression.}
Lasso, SVM, and Huber Regression are well-known optimization problems that can be reformulated as standard quadratic programming problems, where the quadratic term $Q$  is the identity matrix. This structure makes these problems less suited for both the PDHCG and rAPDHG methods. For these three problem types, we set the number of constraints \( m = n \), the sparsity at \( 10^{-4} \), and the penalty parameter $\lambda = 0.01$. 

\begin{table}[ht]
  \centering
  \caption{Solving time (s) over 10 Lasso, SVM and Huber regression problems}
    \begin{tabular}{c|cc|cc|ccc|c|c}
    \hline
    \multirow{2}[4]{*}{n} & \multicolumn{2}{c|}{PDHCG} & \multicolumn{2}{c|}{rAPDHG} & \multicolumn{3}{c|}{SCS} & OSQP  & COPT \\
\cmidrule{2-10}          & GPU   & CPU   & GPU   & CPU   & GPU   & ID CPU & D CPU & CPU   & CPU \\
    \hline
    1E+03 & 0.82  & 0.73  & 0.76  & 0.74  & 0.28  & 1.03  & 0.04  & \bf{0.02}  & 0.04  \\
    1E+04 & 1.28  & 4.98  & 0.94  & 3.36  & 0.58  & 2.19  & 0.40  & \bf{0.40}  & 0.41  \\
    1E+05 & 1.09  & 39.97  & \bf{1.03}  & 16.91  & 14.26  & 20.76  & f     & 1734.80  & 579.50  \\
    1E+06 & \bf{3.65}  & 841.09  & 4.29  & 1272.09  & t     & 557.72  & f     & f     & f \\
    1E+07 & \bf{14.54}  & f     & 16.07  & f     & t     & f     & f     & f     & f \\
    \hline
    \end{tabular}%
  \label{lassosvmmpc}%
\end{table}%

Table \ref{lassosvmmpc} presents the results for 30 randomly generated LASSO, SVM, and Huber regression problems, with 10 instances each and problem dimensions ranging from \(10^3\) to \(10^7\). As seen in previous results, the GPU versions of PDHCG and rAPDHG demonstrate similar performance, with both showing a significant advantage over other solvers in large-scale cases.

\subsection{Real-world Large-scale QP problems} 
\label{realworldqp}
In this section, to further evaluate the capability of our algorithm in solving large-scale problems, we consider the Lasso problem, with data sourced from real-world datasets, specifically LIBSVM \citep{chang2011libsvm} and the SuiteSparse Matrix Collection \citep{davis2011university}. All selected problems contain more than 6,000,000 nonzero elements, with the largest problem containing nearly 400,000,000 nonzeros. The formulation of the Lasso problem is given by:
$$\begin{array}{ll}
  \min_x & \| A x - b \|_2^2 + \lambda \| x \|_1,
\end{array}$$
which is equivalent to the following QP problem when considering its dual form:
$$\begin{array}{ll}
\operatorname{min} & y^T y+\lambda \mathbf{1}^T t \\
\text { s.t. } & y=A x-b \\
& -t \leq x \leq t.
\end{array}$$
In the experiments, we set the last column of matrix $A$ as vector $b$, and we set the parameter $\lambda = 0.01||A^T b||_{\infty}$ for all test cases.

\begin{table}[htbp]
  \centering
   \caption{Solving time (s) over large-scale LASSO problems in LIBSVM and SuiteSparse Matrix Collection.}
   \resizebox{\textwidth}{!}{
   \begin{threeparttable}
    \begin{tabular}{c|ccr|ccccc}
    \toprule
    Problem & m     & n     & \multicolumn{1}{c|}{Sparsity} & PDHCG & rAPDHG  & SCS(GPU) & OSQP  & COPT \\
    \hline
    SLS   & 1,748,122 & 62,729 & 6.21E-05 & \textbf{3.35 } & 7.30  & 345.09  & 80.32  & 88.21  \\
    rcv1\_test & 677,399 & 47,236 & 1.55E-03 & \textbf{7.12 } & 19.54  & f     & f     & f \\
    avazu-site.tr & 23,567,843 & 1,000,000 & 1.50E-05 & \textbf{1377.54 } & 5124.82  & f     & f     & f \\
    avazu-app & 40,428,967 & 1,000,000 & 1.50E-05 & \textbf{1429.55 } & 5557.97  & f     & f     & f \\
    avazu-site & 25,832,830 & 1,000,000 & 1.50E-05 & \textbf{4224.95 } & f     & f     & f     & f \\
    kddb2010\_test & 748,401 & 1,163,024 & 7.74E-06 & \textbf{11.10 } & 46.49  & 490.66  & 255.81  & 69.57  \\
    kdda2010\_test & 510,302 & 20,216,830 & 1.87E-05 & \textbf{61.00 } & 148.01  & f     & f     & f \\
    kddb2010\_train & 19,264,097 & 1,163,024 & 7.97E-06 & \textbf{387.00 } & 1715.50  & f     & f     & f \\
    kdda2010\_train & 8,407,752 & 20,216,830 & 1.80E-06 & \textbf{2705.94 } & f     & f     & f     & f \\
    \hline
    \end{tabular}%
  \label{real_lasso}%
     \begin{tablenotes}
\footnotesize
\item [1] "f" means the algorithm fails to solve the problem within 7200s.
\end{tablenotes}
\end{threeparttable}
}
\end{table}%

Table \ref{real_lasso} presents the solving times of different algorithms on the Lasso problems using selected datasets. The dimensions of matrix $A$ and the corresponding sparsity for each instance are also provided. The results show that PDHCG demonstrates superior performance across nearly all instances. As Table \ref{real_lasso} shown, SCS (GPU), OSQP, and COPT are only able to solve the two smallest instances, $SLS$ and $kddb2010\_test$, while they hit the time limit of 7200 seconds on the larger problems. Compared to rAPDHG, PDHCG exhibits better performance on these real-world problems, achieving an average 4x improvement in solving time. For the largest instances, $avazu$-$site$ and $kdda2010\_train$, only PDHCG successfully finds the solution.


\subsection{Results on Two Benchmark Sets}
\label{twobenchmark}
In this section, we conducted tests using instances from the Maros–Mészáros benchmark \citep{Maros1999qp} and QPLIB relaxations \citep{furini2019qplib}. The Maros–Mészáros dataset, a standard benchmark for convex quadratic programming, consists of 134 problems. The QPLIB dataset includes a variety of quadratic programming problems, some with quadratic and integer constraints. We filtered and relaxed certain instances, selecting 34 problems for our tests. Due to the smaller size of the Maros–Mészáros dataset compared to QPLIB, we set the algorithm's time limit to 600 seconds for the Maros–Mészáros dataset and 3600 seconds for the QPLIB dataset. To assess the stability of our algorithm on small to medium-sized problems, we present the results for the CPU and GPU versions of our PDHCG method, as well as rAPDHG, and include the commercial solver COPT \citep{ge2022cardinal} as a benchmark. For a more comprehensive comparison of other solvers on these two benchmark sets, we refer readers to \citep{lu2023practical}.

Tables \ref{134_total} and \ref{QPLIB_10} present the test results on the Maros–Mészáros benchmark and QPLIB set, respectively, including the number of solved problems and the geometric mean of runtime and iterations. Overall, PDHCG and rAPDHG produce comparable results. For the Maros–Mészáros benchmark dataset, rAPDHG solves more problems, but PDHCG is faster and requires only half the iterations. For the QPLIB set, as the average problem size increases, the GPU implementation of PDHCG begins to show a clear advantage.


\begin{table}[b]
	\centering
	\begin{tabular}{ccccccc}
		\hline & \multicolumn{3}{c}{ \textbf{Tol 1E-03} } & \multicolumn{3}{c}{ \textbf{Tol 1E-06}} \\
		& \textbf{Count} & \textbf{Time} & \textbf{Iteration} & \textbf{Count} & \textbf{Time} &\textbf{Iteration} \\
		\hline \textbf{PDHCG (CPU)} & 124 & 0.99 & 780 & 108 & \bf{1.23}& 2005 \\
		\textbf{PDHCG (GPU)} & 121 & 1.72 & \textbf{763} & 105 & 3.01 &\textbf{1988} \\
		\textbf{rAPDHG (CPU)} & \textbf{128} & \textbf{0.97} & 1714 & 111 & 1.26& 4203 \\
		\textbf{rAPDHG (GPU)} & 127 & 1.94 & 1706 & \bf{114} & 3.20 &4181 \\
       \hline 
        \textbf{COPT} &134&0.08 &-  & 134& 0.08 &- \\
		\hline
	\end{tabular}
 \vspace{2mm}
	\caption{The number of solved instances, the runtime, and the iteration count of different algorithms for solving 134 instances in the Maros–Mészáros benchmark set. The runtime (in seconds) is the geometric average of instances that can be solved by all solvers, with 115 instances at $10^{-3}$ and 86 instances at $10^{-6}$ tolerance.}
	\label{134_total}
\end{table}

Figure \ref{fig} shows the relative KKT error plotted against the number of iterations for two versions of the PDHCG and rAPDHG algorithms on selected problems. The data illustrate that, at certain stages of convergence, both algorithms exhibit linear convergence behavior, in line with theoretical expectations.

\begin{table}[ht]
	\centering
	\begin{tabular}{ccccccc}
		\hline & \multicolumn{3}{c}{ \textbf{Tol 1E-03} } & \multicolumn{3}{c}{ \textbf{Tol 1E-06}} \\
		& \textbf{Count} & \textbf{Time} & \textbf{Iteration} & \textbf{Count} & \textbf{Time} &\textbf{Iteration} \\
		\hline \textbf{PDHCG (CPU)} & 29 & 4.53 & 984 & 19 & 12.01& 3626 \\
		\textbf{PDHCG (GPU)} & \textbf{33} & 2.97 & \textbf{980} & \textbf{24} & \textbf{4.30} &\textbf{3610} \\
		\textbf{rAPDHG (CPU)} & 25 & 4.52 & 2149 & 20 & 16.86& 7238 \\
		\textbf{rAPDHG (GPU)} & 25 & \textbf{2.51} & 2149 & 20 & 5.17 &7238 \\
        \hline \textbf{COPT} & 34& 0.38&-& 34& 0.46 &- \\
		\hline
	\end{tabular}
\vspace{2mm}
	\caption{The number of solved instances, the time, and the iterations of different algorithms for solving 34 instances of the QPLIB with tolerances of $10^{-3}$ and $10^{-6}$. We removed the problem that the number of iterations is less than 10, and 12 instances are counted under two accuracy.}
	\label{QPLIB_10}
\end{table}


\begin{figure}[t]
    \begin{minipage}[t]{0.32\linewidth}
        \centering    
            \includegraphics[width=\linewidth]{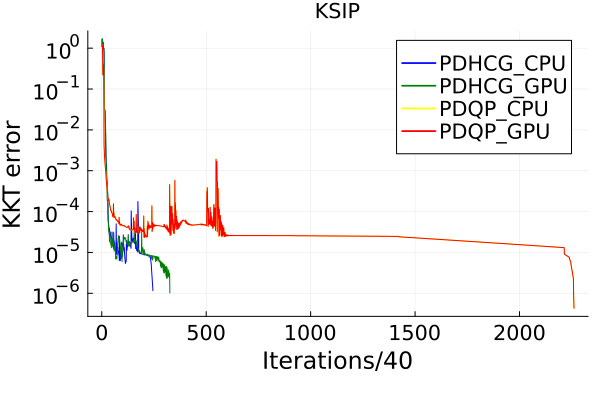}
\end{minipage}
    \hfill
    \begin{minipage}[t]{0.32\linewidth}
            \includegraphics[width=\linewidth]{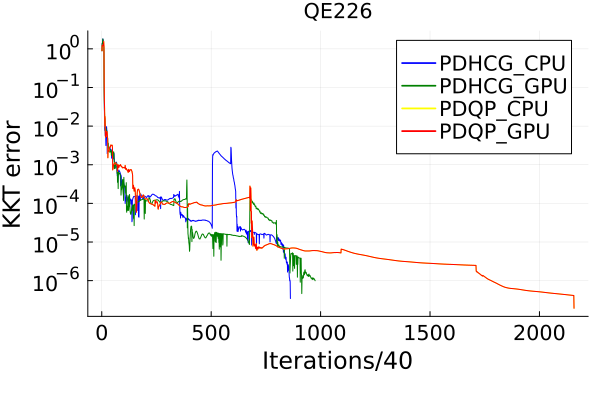}
    \end{minipage}
    \hfill
    \begin{minipage}[t]{0.32\linewidth}
        \centering
            \includegraphics[width=\linewidth]{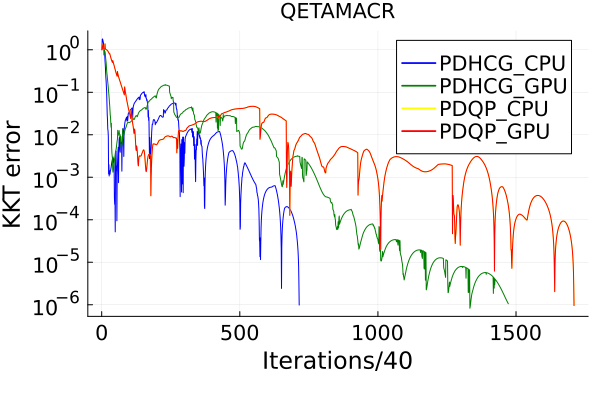}
\end{minipage}
    \hfill
    \begin{minipage}[t]{0.32\linewidth}
            \includegraphics[width=\linewidth]{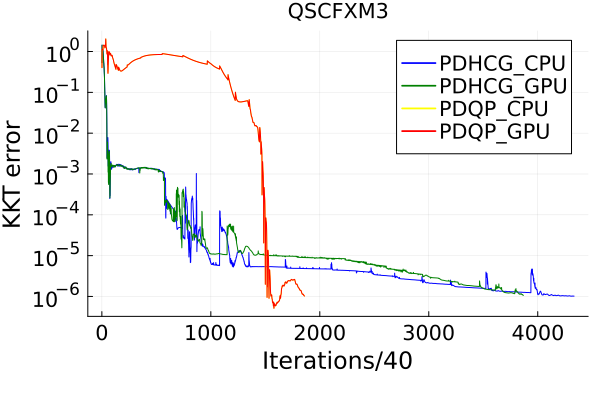}
    \end{minipage}
    \hfill
    \begin{minipage}[t]{0.32\linewidth}
        \centering
            \includegraphics[width=\linewidth]{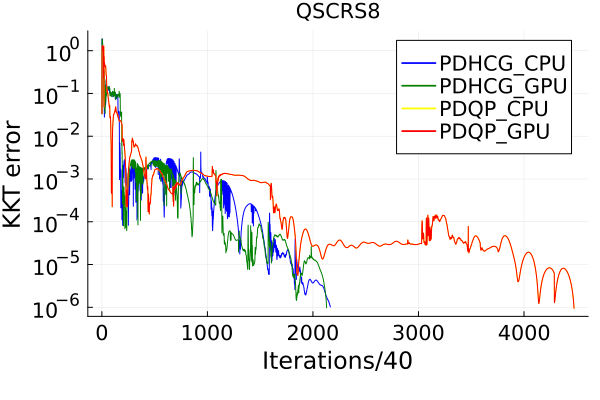}
\end{minipage}
    \hfill
    \begin{minipage}[t]{0.32\linewidth}
            \includegraphics[width=\linewidth]{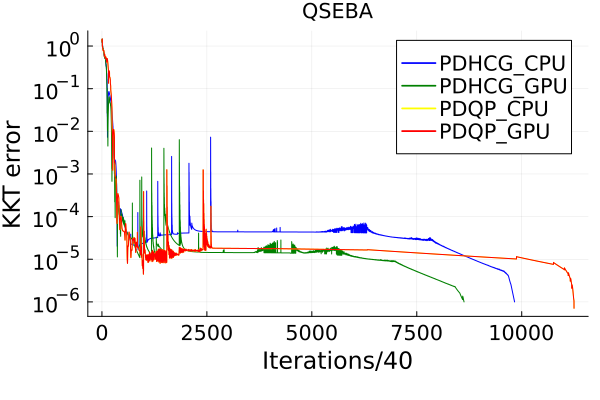}
    \end{minipage}
    \caption{Relative KKT error per iteration for PDHCG and rAPDHG algorithm on specific instances.}
    \label{fig}
\end{figure}

\section{Conclusion}\label{sec:remark}
In this paper, we introduce the restarted PDHCG method for solving convex QP problems. Building on the standard PDHG framework, PDHCG employs conjugate gradient (CG) techniques to solve primal subproblems inexactly, maintaining a linear convergence rate with an improved constant. It is straightforward to implement on GPUs, making it ideal for large-scale problems. Extensive numerical experiments show that PDHCG significantly reduces iteration counts and improves performance compared to rAPDHG, particularly for large-scale QP challenges. The GPU implementation further enhances efficiency and scalability, highlighting its potential for practical applications.


%
%
%
\section*{Acknowledgments}
{We would like to express our sincere gratitude to Weihan Xue (Shanghai Jiao Tong University) and Qiushi Han (University of Illinois Urbana-Champaign) for their invaluable contributions to this research. We are also grateful to Prof. Haihao Lu (Massachusetts Institute of Technology) and Xudong Yin (Cardinal Operations) for their support and assistance throughout the course of this work.}


\bibliographystyle{abbrv}
\bibliography{ref} 
\newpage

\appendix
\section{Proofs of adaptive CG} \label{subsec:proof-adaptive}
First of all, we focus on the iterations within each restart loop. For simplicity of notations, we use $x^k, y^k, z^k$ to denote $x^{n,k}, y^{n,k}, z^{n,k}$, respectively.
\begin{lemma} \label{lemma:ada-1}
  Let $z^k = (x^k, y^{k + 1})$ and $(\hat{x}, \hat{y})$ denote the saddle point of \eqref{prob:general-form}. By choosing  
  \[ 
  \theta = 1, \quad 4 \tau
  \sigma \| A \|^2 < 1, \quad 
  2 \zeta \leq  \min \left\{ \frac{1 - \sqrt {\sigma \tau} \| A \|}{2 \sqrt{\sigma
     \tau} K^2}, \frac{1 - \sqrt{\sigma \tau} \| A \|}{2 \tau K^2} \right\} \]
  and $\{ \varepsilon^k \}^{k = K}_{k = 1}$ being the sequence
  of error terms satisfying \eqref{stop-criteria}, the following holds:
  \begin{itemize}
      \item[1)] For any $0 \leq k \leq K - 1$,
  \[ \frac{\| y^k - \hat{y} \|^2}{2 \sigma} + \frac{\| x^k - \hat{x} \|^2}{2
     \tau} \leq \left( \frac{1 - \tau \sigma \| A \|^2}{2} \right)^{- 1}
     \left( \frac{\| y^0 - \hat{y} \|^2}{2 \sigma} + \frac{\| x^0 - \hat{x}
     \|^2}{2 \tau} \right) \]
  \item[2)] Let $\tilde{x} = \left( \sum^K_{k = 1} x^k \right) / K$ and $\tilde{y} =
  \left( \sum^K_{k = 1} y^k \right) / K$, then
  \[ Q (\tilde{z}, z) \leq \frac{1}{K} \left( \frac{\| y - y^0 \|^2}{2
     \sigma} + \frac{\| x - x^0 \|^2}{2 \tau} \right) \]
     \end{itemize}
\end{lemma}

\begin{proof}
  First of all, the optimality conditions of primal and dual updates are given by 
  \begin{align}\label{eq:opt-cond}
  \begin{split} 
      y^{k + 1} & = (I + \sigma \partial F)^{- 1} (y^k + \sigma A
     \bar{x}), \\
  x^{k + 1}_* &= (I + \tau \partial G)^{- 1} (x^k - \tau A^T  \bar{y})
  \end{split}
  \end{align} 
  where $\bar{x} = x^k + \theta(x^k - x^{k-1}) = 2x^k - x^{k-1}$ and  $\bar{y} = y^{k+1}$. Since we use the inexact primal update, let $\delta^{k+1} = (I + \tau \partial G) (x^{k+1} -x^{k+1}_*) $, we have
  \[
x^{k + 1} = x^{k+1}_* + (I + \tau \partial G)^{-1} \delta^{k+1}
     = (I + \tau \partial G)^{- 1} (x^k - \tau A^T  \bar{y} +
     \delta^{k + 1}).
  \]
  The above equations imply that
\[
      \partial F (y^{k + 1})  = \frac{y^k - y^{k + 1}}{\sigma} + A
     \bar{x}, \quad \text{and} \quad 
     \partial G (x^{k + 1})  = \frac{x^k - x^{k + 1}}{\tau} - A^T
     \bar{y} + \delta^{k + 1}.
\]
  Then, according to the convexity of both $F(\cdot)$ and $G(\cdot)$, for any $(x, y)$, we have
   \begin{align}\label{tylor-F}
      F (y) & \geq F (y^{k + 1}) + \left\langle \frac{y^k -
     y^{k + 1}}{\sigma}, y - y^{k + 1} \right\rangle + \langle A \bar{x}, y -
     y^{k + 1} \rangle\\
    G (x) & \geq G (x^{k + 1}) + \left\langle \frac{x^k - x^{k + 1}}{\tau}, x - x^{k
    + 1} \right\rangle - \langle A (x - x^{k + 1}), \bar{y} \rangle + \langle
    \delta^{k + 1}, x - x^{k + 1} \rangle \label{tylor-G}
  \end{align}
  Sum the above inequalities \eqref{tylor-F} and \eqref{tylor-G} to get
   \begin{align}\begin{split}
  \begin{split}
    \frac{\| y - y^k \|^2}{2 \sigma} + \frac{\| x - x^k \|^2}{2 \tau} &
    \geq [\langle A x^{k + 1}, y \rangle - F (y) + G (x^{k +
    1})] - [\langle A x, y^{k + 1} \rangle - F (y^{k + 1}) + G (x)] \\
    & + \frac{\| y - y^{k + 1} \|^2}{2 \sigma} + \frac{\| y^k - y^{k + 1}
    \|^2}{2 \sigma} + \frac{\| x - x^{k + 1} \|^2}{2 \tau} + \frac{\| x^k -
    x^{k + 1} \|^2}{2 \tau} \\
    & + \langle A (x^{k + 1} - \bar{x}), y^{k + 1} - y \rangle - \langle A
    (x^{k + 1} - x), y^{k + 1} - \bar{y} \rangle + \langle \delta^{k +
    1}, x - x^{k + 1} \rangle \label{ineq-1}
    \end{split}
  \end{split} \end{align}
  Substituting $\bar{x} = 2 x^k - x^{k - 1}$ and $\bar{y} = y^{k + 1}$ into the two terms in the last row of \eqref{ineq-1}, we get
\begin{align}\begin{split}\label{ineq-1-last}
  \begin{split}   
    & \langle A (x^{k + 1} - \bar{x}), y^{k + 1} - y \rangle - \langle A
    (x^{k + 1} - x), y^{k + 1} - \bar{y} \rangle \\
    = & \langle A ((x^{k + 1} - x^k) - (x^k - x^{k - 1})), y^{k + 1} - y
    \rangle \\
    = & \langle A \nobracket (x^{k + 1} - x^k \nobracket), y^{k + 1} - y
    \rangle - \langle A \nobracket (x^k - x^{k - 1} \nobracket), y^k - y
    \rangle - \langle A \nobracket (x^k - x^{k - 1} \nobracket), y^{k + 1} -
    y^k \rangle \\
    \geq & \langle A (x^{k + 1} - x^k), y^{k + 1} - y \rangle - \langle A
    (x^k - x^{k - 1}), y^k - y \rangle - \sqrt{\sigma \tau}
     \| A \|  \frac{\| x^k - x^{k - 1} \|^2}{2 \tau} + \sqrt{\sigma \tau} \| A \|
     \frac{\| y^{k + 1} - y^k \|^2}{2 \sigma}
    \end{split}
  \end{split} \end{align}
  where the last inequality holds because of the Cauchy-Schwarz inequality.
  Substituting \eqref{ineq-1-last}  into \eqref{ineq-1}, we get
   \begin{align}\begin{split}\label{ineq-2}
  \begin{split}
    \frac{\| y - y^k \|^2}{2 \sigma} + \frac{\| x - x^k \|^2}{2 \tau}
   \, \geq \, & [\langle A x^{k + 1}, y \rangle - F (y) + G (x^{k +
    1})] - [\langle A x, y^{k + 1} \rangle - F (y^{k + 1}) + G (x)]\\
    & + \frac{\| y - y^{k + 1} \|^2}{2 \sigma} + \frac{\| y^k - y^{k + 1}
    \|^2}{2 \sigma} + \frac{\| x - x^{k + 1} \|^2}{2 \tau} + \frac{\| x^k -
    x^{k + 1} \|^2}{2 \tau}\\
    & + \langle A (x^{k + 1} - x^k), y^{k + 1} - y \rangle - \langle A (x^k
    - x^{k - 1}), y^k - y \rangle\\
    & - \sqrt{\sigma \tau} \| A \|  \frac{\| x^k - x^{k - 1} \|^2}{2 \tau} -
    \sqrt{\sigma \tau} \| A \| \frac{\| y^{k + 1} - y^k \|^2}{2 \sigma} + \langle
    \delta^{k + 1}, x - x^{k + 1} \rangle 
    \end{split}
  \end{split} \end{align}
  By taking the sum of \eqref{ineq-2} from $k = 0$ to $K - 1$, for any $(x, y)$, we have
   \begin{align}\begin{split} \label{sum-ineq-2} 
  \begin{split}
    & \frac{\| y - y^0 \|^2}{2 \sigma} + \frac{\| x - x^0 \|^2}{2 \tau} +
    \langle A (x^K - x^{K - 1}), y^K - y \rangle \\
    \geq  & \sum^K_{k = 1} [\langle A x^k, y \rangle - F (y) + G
    (x^k)] - [\langle A x, y^k \rangle - F (y^k) + G (x)]  +  \frac{\| y - y^K \|^2}{2 \sigma} + \frac{\| x - x^K \|^2}{2 \tau} \\
    & + (1
    - \sqrt{\sigma \tau} \| A \|) \sum^K_{k = 1} \left( \frac{\| y^k - y^{k - 1} \|^2}{2
    \sigma}
    + \frac{\| x^k - x^{k -
    1} \|^2}{2 \tau} \right)  + \frac{\| x^K - x^{K - 1} \|^2}{2 \tau} + \sum^{K -
    1}_{n = 0} \langle \delta^{k + 1}, x - x^{k + 1} \rangle 
    \end{split}
  \end{split} \end{align}
  According to the Cauchy-Schwartz inequality, we have
  \[ \langle A (x^K - x^{K - 1}), y^K - y \rangle \leq \frac{\| x^K -
     x^{K - 1} \|^2}{4 \tau} + 2 \sigma \tau \| A \|^2 \frac{\| y - y^K
     \|^2}{2 \sigma}. \]
  Substitute it into \eqref{sum-ineq-2} to get
   \begin{align}\begin{split}
    \frac{\| y - y^0 \|^2}{2 \sigma} + \frac{\| x - x^0 \|^2}{2 \tau} &
    \geq \sum^K_{k = 1} [\langle A x^k, y \rangle - F (y) + G
    (x^k)] - [\langle A x, y^k \rangle - F (y^k) + G (x)]\\
    & + (1 - 2\sigma \tau \| A \|^2) \frac{\| y - y^K \|^2}{2 \sigma} +
    \frac{\| x - x^K \|^2}{2 \tau} + (1 - \sqrt{\sigma \tau} \| A \|) \sum^K_{k = 1}
    \frac{\| y^k - y^{k + 1} \|^2}{2 \sigma} \\
    & + (1 - \sqrt{\sigma \tau} \| A \|) \sum^{K - 1}_{k = 1} \frac{\| x^k - x^{k- 1} \|^2}{2 \tau} + \sum^K_{k = 1} \langle \delta^k, x - x^k \rangle + \frac{\| x^K - x^{K - 1} \|^2}{4 \tau}\label{sum-ineq-3}
  \end{split} \end{align}
  Since $\delta^k = (I+\tau \partial G) (x^k - x^k_*)$ and $\|x^k - x^k_*\|\leq \varepsilon^k$ satisfies \eqref{stop-criteria}, we have $\delta^1 \leq \zeta \| y^1 - y^0 \|$ and $ \delta^k \leq \delta^{k-1} +  \zeta \| z^{k-1} - z^{k-2} \|$, so for any $1 \leq m \leq K$, we have
   \begin{align}\begin{split} 
      \left\| \sum_{k = 1}^m \delta^k \right\|^2 \leq  K \sum_{k =
     1}^K \| \nobracket \delta^k | |^2 \leq \zeta^2 K^2 \left\| \sum^{K - 1}_{k = 1} \| z^k - z^{k - 1} \| +\| y^1 - y^0 \| \right\|^2 
     \leq  4 \zeta^2 K^3 \sum^{K - 1}_{k = 1} \| z^k - z^{k - 1} \|^2\label{eps-ineq}
  \end{split} \end{align}
where the last inequality use the fact that $\| y^1 - y^0 \| \leq \| z^1 - z^0 \|$. Then, for the last error term in \eqref{sum-ineq-3},  we have
   \begin{align}\begin{split}
    & \sum^K_{n = 1} \langle \delta^n, x - x^n \rangle  \\
    = & \sum^K_{n = 1}
    \langle \delta^n, (x - x^K) + (x^K - x^{K - 1}) + \cdots + (x^{n + 1}
    - x^n) \rangle  \\
    = & \left\langle \sum^K_{n = 1} \delta^n, x - x^K \right\rangle +
    \sum^{K - 1}_{n = 1} \left\langle \sum^n_{m = 1} \delta^m, x^{n + 1}
    - x^n \right\rangle \\
    \geq & - \frac{\| x - x^K \|^2}{4 \tau} - \tau \left\| \sum^K_{n =
    1} \delta^n \right\|^2 - (1 - \sqrt{\sigma \tau} \| A \|) \sum^{K - 1}_{n =
    1} \frac{\| x^{n + 1} - x^n \|^2}{4 \tau}  - \frac{\tau}{1 - \sqrt{\sigma \tau} \| A \|} \sum^{K - 1}_{n = 1} \left\|
    \sum^n_{m = 1} \delta^m \right\|^2 \\
    \geq &  - \frac{\| x - x^K \|^2}{4 \tau} - (1 - \sqrt{\sigma \tau} \| A \|) \sum^{K - 1}_{n = 1} \frac{\| x^{n + 1} - x^n \|^2}{4 \tau}  - \frac{\tau}{1 - \sqrt{\sigma \tau} \| A \|} \sum^K_{n = 1} \left\| \sum^n_{m = 1} \delta^m \right\|^2 \label{err-ineq-1}
    \end{split} \end{align}
where the first inequality holds because of Cauchy-Schwartz inequality. Now apply the inequality \eqref{eps-ineq} to the last term in the right-hand-side of \eqref{err-ineq-1}
     \begin{align}\begin{split}
    \sum^K_{n = 1} \langle \delta^n, x - x^n \rangle &  \geq - \frac{\| x - x^K \|^2}{4 \tau} - (1 - \sqrt{\sigma \tau} \| A \|) \sum^{K - 1}_{n = 1} \frac{\| x^{n + 1} - x^n \|^2}{4 \tau}  - \frac{4 \tau\zeta^2 K^3}{1 - \sqrt{\sigma \tau} \| A \|} \sum^K_{n = 1} 
    \sum^{K - 1}_{n = 1} \| z^n - z^{n - 1} \|^2 \\
    & \geq - \frac{\| x - x^K \|^2}{4 \tau} - (1 - \sqrt{\sigma \tau} \| A \|) \sum^{K - 1}_{n = 1} \frac{\| x^{n + 1} - x^n \|^2}{4 \tau} -
    \frac{4 \tau \zeta^2 K^4}{1 -\sqrt{\sigma \tau} \| A \|} \sum^{K - 1}_{k = 1} \| z^n -z^{n - 1} \|^2 \\
    & \geq - \frac{\| x - x^K \|^2}{4 \tau} - (1 - \sqrt{\sigma \tau} \| A \|) \sum^K_{n = 2} \frac{\| x^n - x^{n - 1} \|^2}{4 \tau} \\
    &  - (1 - \sqrt{\sigma \tau} \| A \|) \sum^{K - 1}_{n = 1} \frac{\| x^n - x^{n
    - 1} \|^2}{4 \tau} - (1 - \sqrt{\sigma \tau} \| A \|) \sum^K_{n = 2} \frac{\| y^n
    - y^{n - 1} \|^2}{4 \sigma}  \\
    & \geq - \frac{\| x - x^K \|^2}{4 \tau} - (1 - \sqrt{\sigma \tau} \| A \|) \sum^K_{n = 1} \frac{\| x^n - x^{n - 1} \|^2}{2 \tau} - (1 - \sqrt{\sigma \tau} \| A \|) \sum^K_{n = 1} \frac{\| y^n - y^{n - 1} \|^2}{4 \sigma}. \label{eps-ineq-1}
  \end{split} \end{align}
  where the third inequality holds because $2 \zeta \leq  \min \left\{ \frac{1 - \sqrt{\sigma \tau} \| A \|}{2 \sqrt{\sigma
     \tau} K^2}, \frac{1 - \sqrt{\sigma \tau} \| A \|}{2 \tau K^2} \right\}.$
  Plugging \eqref{eps-ineq-1} into \eqref{sum-ineq-3}, we have
   \begin{align}\begin{split}
    \frac{\| y - y^0 \|^2}{2 \sigma} + \frac{\| x - x^0 \|^2}{2 \tau} 
    \geq & \sum^K_{k = 1} [\langle A x^k, y \rangle - F (y) +
    G (x^k)] - [\langle A x, y^k \rangle - F (y^k) + G (x)] \\
    & + (1 - 2 \sigma \tau \| A \|^2) \frac{\| y - y^K \|^2}{2 \sigma} +
    \frac{\| x - x^K \|^2}{2 \tau} + (1 - \sqrt{\sigma \tau} \| A \|) \sum^K_{k = 1}
    \frac{\| y^k - y^{k - 1} \|^2}{2 \sigma} \\
    & + (1 - \sqrt{\sigma \tau} \| A \|) \sum^K_{k = 1} \frac{\| x^k - x^{k - 1}
    \|^2}{2 \tau} - \frac{\| x - x^K \|^2}{4 \tau} \\
    & - (1 - \sqrt{\sigma \tau} \| A \|) \sum^K_{k = 1} \frac{\| x^k - x^{k - 1}
    \|^2}{2 \tau} - (1 - \sqrt{\sigma \tau} \| A \|) \sum^K_{k = 1} \frac{\| y^k -
    y^{k - 1} \|^2}{4 \sigma}. \label{sum-ineq-4}
  \end{split} \end{align}
  Letting $(x, y) = (\hat{x}, \hat{y})$ be the saddle point of \eqref{prob:general-form}, we have
  $$Q (z^k, \hat{z}) = [\langle A x^k,
  \hat{y} \rangle - F (\hat{y}) + G (x^k)] - [\langle A \hat{x}, y^k
  \rangle - F (y^k) + G (\hat{x})] \geq 0.
  $$
  Choose $2 \sigma \tau \| A \|^2\leq 1$, then \eqref{sum-ineq-4} becomes
  \[ \frac{\| \hat{y} - y^0 \|^2}{2 \sigma} + \frac{\| \hat{x} - x^0 \|^2}{2
     \tau} \geq (1 - 2\sigma \tau \| A \|^2) \frac{\| \hat{y} - y^K
     \|^2}{2 \sigma} + \frac{\| \hat{x} - x^K \|^2}{4 \tau}, \]
  which implies that for any $k$
  \[ \frac{\| y^k - \hat{y} \|^2}{2 \sigma} + \frac{\| x^k - \hat{x} \|^2}{2
     \tau} \leq \left( \frac{1 - 2\tau \sigma \| A \|^2}{2} \right)^{- 1}
     \left( \frac{\| y^0 - \hat{y} \|^2}{2 \sigma} + \frac{\| x^0 - \hat{x}
     \|^2}{2 \tau} \right). \]
  So we complete the proof of part $1)$. Finally, by choosing $\tilde{x} = \left(
  \sum^K_{k = 1} x^k \right) / K$ and $\tilde{y} = \left( \sum^K_{k = 1} y^k
  \right) / K$ in \eqref{sum-ineq-4} and using the convexity of $G$ and $F$, we have
  \[ Q (\tilde{z}, z) = [\langle A \tilde{x}, y \rangle - F (y) + G
     (\tilde{x})] - [\langle A x, \tilde{y} \rangle - F (\tilde{y}) +
     G (x)] \leq \frac{1}{K} \left( \frac{\| y - y^0 \|^2}{2 \sigma} +
     \frac{\| x - x^0 \|^2}{2 \tau} \right), \]
  which completes the proof of part $2)$.
\end{proof}

\begin{lemma} \label{lemma:ada-2}
  Suppose the restart frequency $K$ is fixed. Let $\sigma = \tau =
  \frac{1}{2 \| A \|}$, then we have
  \begin{itemize}
      \item[1)]for any $z \in \mathcal{Z}$,
  \[ Q (\tilde{z}, z) \leq \frac{\| A \|}{K} \| z^0 - z \|^2; \]
  \item[2)]for any $0 \leq k \leq K - 1$, 
  \[ \| z^k - \hat{z} \|^2 \leq 4 \| z^0 - \hat{z} \|^2. \]
  \end{itemize}
\end{lemma}

\begin{proof}
  Firstly,  since $\sigma, \tau$ is chosen such that $4 \sigma \tau \| A \|^2\leq 1$, part 2) in \Cref{lemma:ada-1} implies that
  \[ Q (\tilde{z}, z) \leq \frac{1}{K} \left( \frac{\| y - y^0 \|^2}{2
     \sigma} + \frac{\| x - x^0 \|^2}{2 \tau} \right) \leq \frac{\| A
     \|}{K} \| z^0 - z \|^2. \]
  Besides, it is easy to see that 
  \[ \frac{\| y^k - \hat{y} \|^2}{2 \sigma} + \frac{\| x^k - \hat{x} \|^2}{2
     \tau} = \| A \| \| z^k - \hat{z} \|^2 \]
     and 
  \[ \left( \frac{1 - 2\tau \sigma \| A \|^2}{2} \right)^{- 1} \left( \frac{\|
     y^0 - \hat{y} \|^2}{2 \sigma} + \frac{\| x^0 - \hat{x} \|^2}{2 \tau}
     \right) = 4 \| A \| \| z^0 - \hat{z} \|^2. \]
  So part 1) in \Cref{lemma:ada-1} implies that
  \[ \| z^k - \hat{z} \|^2 \leq 4 \| z^0 - \hat{z} \|^2. \]
\end{proof}

\begin{lemma} \label{lemma:ada-3}
  Suppose the restart frequency $K$ satisfies $K \geq
  \frac{4 \| A \|}{\xi}$. Then, it holds for any $\dot{z} \in \mathcal{Z}$ that
  \[ G_{\xi} (\tilde{z}, \dot{z}) \leq \frac{2 \| A \|}{K} \| z^0 -
     \dot{z} \|^2. \]
\end{lemma}

\begin{proof}
  By \Cref{lemma:ada-2}, for any $z \in \mathcal{Z}$, we have
  \[ G_{\xi} (\tilde{z}, \dot{z}) = Q (\tilde{z}, z) - \xi \| z - \dot{z} \|^2 \leq \frac{\| A \|}{K} \|
     z^0 - z \|^2 - \frac{\xi}{2} \| z - \dot{z} \|^2 \]
  With the choice $K \geq \frac{4 \| A \|}{\xi}$, the right hand side
  attains its maximum at $z = \frac{\frac{\xi}{2} \dot{z} - \frac{\| A \|}{K}
  z^0}{\frac{\xi}{2} - \frac{\| A \|}{K}}$. So we have
  \[ G_{\xi} (\tilde{z}, \dot{z}) \leq \frac{\xi \frac{\| A
     \|}{K}}{\xi - \frac{2 \| A \|}{K}} \| z^0 - \dot{z} \|^2 \leq \frac{2 \| A \|}{K} \| z^0 -
     \dot{z} \|^2 \]
  where the last inequality holds because $\xi - \frac{2 \| A \|}{K} \geq \frac{\xi}{2}$. So we complete the proof.
\end{proof}

Now we are ready to prove \Cref{thm:ada}.

\begin{proof}[Proof of \Cref{thm:ada}]
  We complete the proof via mathematical induction. The result holds for $n = 0$. Suppose it holds for all $n
  \leq D$. Define $z^{\ast}_n = \tmop{argmin}_{z \in \mathcal{Z}^{\ast}} \| z^{n, 0} -
  z \|^2$, then we have
   \[\begin{aligned}
    \| z^{D + 1, 0} - z^{0, 0} \| \leq & \sum_{n = 0}^D \| z^{n + 1, 0} -
    z^{n, 0} \| \leq \sum_{n = 0}^D \| z^{n + 1, 0} - z^{\ast}_n \| + \|
    z^{\ast}_n - z^{n, 0} \| \leq \sum_{n = 0}^D 3 \| z^{\ast}_n - z^{n,
    0} \|  \\
    = & \sum_{n = 0}^D 3 \tmop{dist} (z^{n, 0} , \mathcal{Z}^{\ast})
    \leq 3 \sum_{n = 0}^D e^{- n} \tmop{dist} (z^{0, 0} ,
    \mathcal{Z}^{\ast}) \leq \frac{3}{1 - 1 / e} \tmop{dist} (z^{0, 0},
    \mathcal{Z}^{\ast})
\end{aligned}
  \]
  where the first and second inequality hold because of the triangle inequality, the third inequality use part 2) in \Cref{lemma:ada-2}, and the fourth inequality holds because the mathematical induction assume \Cref{thm:ada} holds for all $n \leq D$. This implies that $z^{D + 1, 0}$ is in set $B_R (z^{0, 0})$ and thus
  $\alpha_{\xi}$ is a valid constant of quadratic growth for iterates $z^{D +
  1, 0}$. Thus, we have
   \[\begin{aligned}
    \tmop{dist}^2 (z^{D + 1, 0}, \mathcal{Z}^{\ast}) & \leq
    \frac{1}{\alpha_{\xi}} G_{\xi} (z^{D + 1, 0}, z^{\ast}_D) =
    \frac{1}{\alpha_{\xi}} G_{\xi} (\tilde{z}^{D, K}, z^{\ast}_D)  \\
    & \leq \frac{2 \| A \|}{\alpha_{\xi} K} \| z^{D, 0} - z^{\ast}_D \|^2
    = \frac{2 \| A \|}{\alpha_{\xi} K} \tmop{dist}^2 (z^{D, 0},
    \mathcal{Z}^{\ast})  \\
    & \leq e^{- 2} \tmop{dist}^2 (z^{D, 0}, \mathcal{Z}^{\ast})
    \leq e^{- 2 (D + 1)} \tmop{dist}^2 (z^{0, 0}, \mathcal{Z}^{\ast})
\end{aligned}
  \]
where the first inequality holds because of the quadratic growth property, the second inequality holds because of \Cref{lemma:ada-3}, and the third inequality holds because of the choice of $K$. Finally, by induction we complete the proof of the boundedness and linear convergence of $\{ z^{n,
  0} \}$.
\end{proof}

\begin{proof}[Proof of Proposition \ref{cor:ada}]
Suppose the iteration number need to satisfy the stop criteria in the $k$-th inner loop is $\rho_k$, according to the convergence behaviour of CG, we have
   \begin{align}\begin{split}
    \| x^k - x_{\ast}^k \| \leq & 2r^{\rho_k} \| x^{k - 1} - x_{\ast}^k \|
     \\
    \leq & 2 r^{\rho_k} \left(\| x^{k - 1} - x_{\ast}^{k - 1} \| + \| x_{\ast}^k
    - x_{\ast}^{k - 1} \|\right)  \\
    \leq & 2 r^{\rho_k} \left( \varepsilon^{k - 1} + \| (I + \tau \partial G)^{- 1} \| \cdot 
    \| (x^{k - 1} - x^{k - 2}) - \tau A^T (y^k - y^{k - 1}) \| \right)
     \\
    \leq & 2r^{\rho_k} \left( \varepsilon^{k - 1} +
    (1 + \tau \|A \|)  \cdot  \| z^{k - 1} - z^{k - 2} \|\right)
  \end{split} \end{align}
where the second inequality holds because of the triangle inequality, the third inequality holds because of the optimility condition \eqref{eq:opt-cond}, and the last one holds because $ \| (I + \tau \partial G)^{- 1} \|\leq 1$. The stop criteria of inexact CG is $ \| x^k - x_{\ast}^k \| \leq \varepsilon^k$, so we need to guarantee that 
  \[ 2r^{\rho_k} \left( \varepsilon^{k - 1} + \| (I + \tau \partial G)^{- 1} \|\cdot 
    (1 + \tau \|A \|)  \cdot  \| z^{k - 1} - z^{k - 2} \|\right) \leq \varepsilon^k. \]
Since $\varepsilon^k  = \varepsilon^{k-1} + \frac{\zeta \| z^{k-1} - z^{k-2} \|}{1 + \tau \| Q\|}$, we have
\begin{align*}
    r^{\rho_k} \leq \frac{ \varepsilon^{k-1} + \frac{\zeta \| z^{k-1} - z^{k-2} \|}{1 + \tau \| Q\|}}{2 \left( \varepsilon^{k - 1} +  
    (1 + \tau \|A \|)  \cdot  \| z^{k - 1} - z^{k - 2} \|\right)} \leq  \frac{ \zeta }{2  
    (1 + \tau \|A \|)(1 + \tau \| Q\|)  }
\end{align*}
  This holds for any iteration $k$, so the iteration number needed
  is bounded by the constant:
  \[ \log_r \frac{\zeta}{ 2  
    (1 + \tau \|A \|)(1 + \tau \| Q\|)  } \]
\end{proof}

\section{Proofs of fixed CG} \label{subsec:proof-fixed}
First of all, we focus on the iterations within each restart loop. For simplicity of notations, we use $x^k, y^k, z^k$ to denote $x^{n,k}, y^{n,k}, z^{n,k}$, respectively.
\begin{lemma} \label{lemma:fix-1}
  Let $z^k = (x^k, y^k)$ be the iterates generated by restarted PDHCG. Suppose
  that the iteration number of primal CG satisfies $N \geq \log_{\gamma_K} \frac{1}{K}$. For $0
  \leq t \leq K$, we choose $\sigma_k, \tau_k, \beta_k$ such that the following identity holds
  \[ \frac{\sigma_k}{\sigma_{k + 1}} =
     \frac{\tau_k }{\tau_{k + 1}} =\theta_{k+1} = \frac{k+1}{k+2}, \, \beta_k = k. \]
     Denote $\eta_k = k+1, 0\leq k\leq K$. Then,
\begin{itemize}
    \item[1)] it holds for any $0 \leq k \leq K$ that
   \[
    \frac{1}{2 \tau_k}  \| x - x^{k + 1} \|^2 - \frac{1}{4 \sigma_k} \| y -
    y^{k + 1} \|^2 \leq \frac{1}{2
    \tau_k} \| x - x^0 \|^2 + \frac{1}{2 \sigma_k} \| y - y^0 \|^2; \]
    \item[2)] it holds for any $0 \leq k \leq K$ and $\tilde{z }^{k + 1} = (\tilde{x }^{k + 1}, \tilde{y }^{k + 1})$ that
    \[
    \beta_k Q (\tilde{z }^{k + 1}, z) \leq  \frac{1}{2 \tau_k} \| x - x^0 \|^2 + \frac{1}{2 \sigma_k} \| y - y^0
    \|^2. \]
\end{itemize}
\end{lemma}

\begin{proof}
Recall the definition of duality gap and lagrangian function \eqref{eq:primal-dual-formulation}. We have
   \begin{align}\label{main-ineq} \begin{split}
    & \beta_k Q (\tilde{z }^{k + 1}, z) - (\beta_k - 1) Q (\tilde{z }^k,z) \\
    = & \beta_k G (\tilde{x }^{k + 1}) - (\beta_k - 1) G
    (\tilde{x }^k) - G (x) + \beta_k F (\tilde{y }^{k + 1}) - (\beta_k - 1) F (\tilde{y
    }^k) - F (y) \\
    &  + \langle A (\beta_k \tilde{x }^{k + 1} - (\beta_k - 1)
    \tilde{x }^k), y \rangle - \langle A x, \beta_k \tilde{y}^{k + 1} -(\beta_k - 1) \tilde{y }^k \rangle \\
    \leq & G (\beta_k \tilde{x }^{k + 1} - (\beta_k - 1)
    \tilde{x }^k) - G (x) + F (\beta_k \tilde{y }^{k + 1} - (\beta_k - 1) \tilde{y
    }^k) - F (y) + \langle A x^{k + 1}, y \rangle - \langle A x, y^{k + 1} \rangle \\
    = & G (x^{k + 1}) - G (x) + F (y^{k + 1}) - F (y)  + \langle A
    x^{k + 1}, y \rangle - \langle A x, y^{k + 1} \rangle
  \end{split} \end{align}
where the first inequality holds because of the convexity of $F$ and $G$ and the last equality holds because of the fact $\tilde{x}^{k + 1} = (1 - \beta^{- 1}_k) \tilde{x}^k +
  \beta^{- 1}_k x^{k + 1}$ and $\tilde{y}^{k + 1} = (1 - \beta^{- 1}_k)
  \tilde{y}^k + \beta^{- 1}_k y^{k + 1}$. 
Since $y^{k+1}$ is the global minimal of the strong convex problem $F(y) - \langle A \bar{x}^k, y \rangle + \frac{1}{2\sigma_k}\|y - y^k\|^2$, we have
  \[ F(y) + \frac{1}{2\sigma_k}\|y - y^k\|^2 - \langle A \bar{x}^k, y - y^{k + 1}  \rangle - F (y^{k + 1}) - \frac{1}{2\sigma_k}\|y^{k+1} - y^k\|^2 \geq  \frac{1}{2 \sigma_k} \|
     y - y^{k + 1} \|^2, \]
or equivalently,
   \begin{align}
    F (y^{k + 1}) - F (y) \leq & \frac{1}{2 \sigma_k} \| y - y^k \|^2 -
    \frac{1}{2 \sigma_k} \| y - y^{k + 1} \|^2 - \frac{1}{2 \sigma_k} \| y^{t
    + 1} - y^k \|^2  + \langle A \bar{x}^k, y^{k + 1} - y \rangle. \label{dual-opt}
 \end{align}
Let $x^{k+1}_*$ denote exact solution of the primal update, i.e.,
  \begin{equation} \label{def-H}
      x^{k + 1}_{\ast} = \tmop{argmin}_x H^{k + 1} (x) = \tmop{argmin}_x G (x)
     + \langle A x, y^{k + 1} \rangle + \frac{1}{2 \tau_k} \| x - x^k \|^2
  \end{equation}
The primal solution after $N$ steps of CG is denoted by $x^{k + 1}$. Then, the L-smooth property of $H(\cdot)$ implies that
   \begin{align}\begin{split}
    H (x^{k + 1}) & \leq H (x^{k + 1}_{\ast}) + \left( \frac{1}{2 \tau_k}
    + \frac{\| Q \|}{2} \right)  \| x^{k + 1} - x^{k + 1}_{\ast} \|^2
     \\
    & \leq H (x) - \frac{1}{2 \tau_k}  \| x - x^{k + 1}_{\ast} \|^2 +
    \left( \frac{1}{2 \tau_k} + \frac{\| Q \|}{2} \right)\cdot 4 \gamma_k^{2 N} \|
    x^k - x^{k + 1}_{\ast} \|^2 \label{primal-bound}
  \end{split} \end{align}
where the second inequality holds because of the linear convergence of CG algorithm. Based on the definition in \eqref{def-H}, inequality \eqref{primal-bound} is equivalent to
   \begin{align}\begin{split}
    G (x^{k + 1}) \leq & G (x) + \frac{1}{2 \tau_k} \| x - x^k \|^2 -
    \frac{1}{2 \tau_k} \| x^{k + 1} - x^k \|^2  - \frac{1}{2 \tau_k}  \| x - x^{k + 1}_{\ast} \|^2 \\
    &  + \left( \frac{1}{2
    \tau_k} + \frac{\| Q \|}{2} \right) \cdot 4 \gamma_k^{2 N} \| x^k - x^{k +
    1}_{\ast} \|^2 + \langle A (x - x^{k + 1}), y^{k + 1} \rangle \label{primal-bound-2}
  \end{split} \end{align}
We bound the terms $\| x^k - x^{k +
    1}_{\ast} \|^2$ and $-\| x - x^{k +
    1}_{\ast} \|^2$ in the following way:
\[
    \| x^k - x^{k + 1}_{\ast} \|^2 \leq 2 \| x^k - x^{k + 1} \|^2 + 2 \|
    x^{k + 1} - x^{k + 1}_{\ast} \|^2 = 2 \| x^k - x^{k + 1} \|^2 + 8
    \gamma_k^{2 N} \| x^k - x^{k + 1}_{\ast} \|^2,
\]
which implies that
  \begin{equation} \label{ast-bound-1}
      \| x^k - x^{k + 1}_{\ast} \|^2 \leq \frac{2}{1 - 8 \gamma_k^{2 N}} 
     \| x^k - x^{k + 1} \|^2.
  \end{equation}
  Then, for any $x \in \mathbb{R}^n$, the following inequality holds for any $p_k > 0$,
  \[
   \begin{aligned}
    \| x - x^{k + 1} \|^2 \leq & (1 + p_k) \| x - x^{k + 1}_{\ast} \|^2 +
    \left( 1 + \frac{1}{p_k} \right)  \| x^{k + 1} - x^{k + 1}_{\ast} \|^2
     \\
    \leq  & (1 + p_k) \| x - x^{k + 1}_{\ast} \|^2 + 4\left( 1 + \frac{1}{p_k}
    \right) \gamma_k^{2 N} \| x^k - x^{k + 1}_{\ast} \|^2  \\
    \leq & (1 + p_k) \| x - x^{k + 1}_{\ast} \|^2 + 4\left( 1 +
    \frac{1}{p_k} \right)  \frac{8 \gamma_k^{2 N}}{1 - 8 \gamma_k^{2 N}} \|
    x^k - x^{k + 1} \|^2  
\end{aligned}
\]
where the first inequality holds because of the Cauchy-Schwarz inequality, the second inequality holds because of the linear convergence property of CG method, and the last one holds owing to \eqref{ast-bound-1}. The above inequality implies that
  \begin{equation} \label{ast-bound-2}
      - \| x - x^{k + 1}_{\ast} \| \leq - \frac{1}{1 + p_k} \| x - x^{t +
     1} \|^2 + \frac{1}{p_k} \frac{8 \gamma_k^{2 N}}{1 - 8 \gamma_k^{2 N}} \|
     x^k - x^{k + 1} \|^2.
  \end{equation}
  Combining \eqref{ast-bound-1}, \eqref{ast-bound-2}, and \eqref{primal-bound-2} together, we obtain
   \begin{align}\begin{split}
    G (x^{k + 1}) \leq & G (x) + \frac{1}{2 \tau_k} \| x - x^k \|^2 -
    \frac{1}{2 \tau_k} \| x^{k + 1} - x^k \|^2   - \frac{1}{2 \tau_k (1 + p_k)}  \| x - x^{k + 1} \|^2  \\
    & + \frac{4 \gamma_k^{2 N}}{p_k \tau_k (1 - 8 \gamma_k^{2 N})} \| x^k - x^{t +
    1} \|^2 + \left( \frac{1}{\tau_k} + \| Q \| \right)  \frac{4 \gamma_k^{2 N}}{1 -
    8 \gamma_k^{2 N}} \| x^k - x^{k + 1} \|^2 + \langle A (x - x^{k + 1}),
    y^{k + 1} \rangle  \\
    = & G (x) + \frac{1}{2 \tau_k} \| x - x^k \|^2 - \frac{1}{2 \tau_k (1 +
    p_k)}  \| x - x^{k + 1} \|^2  \\
    & - \left( \frac{1}{2 \tau_k} - \frac{4(p_k + 1) \gamma_k^{2 N}}{p_k
    \tau_k (1 - 8 \gamma_k^{2 N})} - \frac{4\gamma_k^{2 N} \| Q \|}{1 - 8
    \gamma_k^{2 N}} \right) \| x^k - x^{k + 1} \|^2 + \langle A (x - x^{t +
    1}), y^{k + 1} \rangle. \label{primal-bound-3}
  \end{split} \end{align}
 Substitute \eqref{dual-opt} and \eqref{primal-bound-3} into \eqref{main-ineq} to get
   \begin{align}\begin{split}
    & \beta_k Q (\tilde{z }^{k + 1}, z) - (\beta_k - 1) Q (\tilde{z }^k,
    z) \\
    \leq & G (x^{k + 1}) - G (x) + F (y^{k + 1}) - F (y) + \langle A
    x^{k + 1}, y \rangle - \langle A x, y^{k + 1} \rangle  \\
    \leq & \frac{1}{2 \tau_k} \| x - x^k \|^2 - \frac{1}{2 \tau_k (1 +
    p_k)}  \| x - x^{k + 1} \|^2 + \frac{1}{2 \sigma_k} \| y - y^k \|^2 - \frac{1}{2 \sigma_k} \| y -
    y^{k + 1} \|^2 - \frac{1}{2 \sigma_k} \| y^{k + 1} - y^k \|^2  \\
    & - \left( \frac{1}{2 \tau_k} - \frac{4(p_k + 1) \gamma_k^{2 N}}{p_k
    \tau_k (1 - 8 \gamma_k^{2 N})} - \frac{4 \gamma_k^{2 N} \| Q \|}{1 - 8
    \gamma_k^{2 N}} \right) \| x^{k + 1} - x^k \|^2 \\
    & + \langle A (x - x^{k + 1}), y^{k + 1} \rangle + \langle A \bar{x}^k,
    y^{k + 1} - y \rangle  + \langle A x^{k + 1}, y \rangle - \langle A x, y^{k + 1} \rangle.
    \label{main-ineq-2}
  \end{split} \end{align}
  For the terms in the last rows of \eqref{main-ineq-2}, we have
   \begin{align}\begin{split}
    & \langle A (x - x^{k + 1}), y^{k + 1} \rangle + \langle A \bar{x}^k,
    y^{k + 1} - y \rangle + \langle A x^{k + 1}, y \rangle - \langle A x, y^{k
    + 1} \rangle  \\
    = & \langle A (x - x^{k + 1}), y^{k + 1} \rangle + \langle A (\theta_k
    (x^k - x^{k - 1}) + x^k), y^{k + 1} - y \rangle + \langle A x^{k + 1}, y
    \rangle - \langle A x, y^{k + 1} \rangle  \\
    = & \langle A (x^{k + 1} - x^k), y - y^{k + 1} \rangle - \theta_k \langle
    A (x^k - x^{k - 1}), y - y^k \rangle - \theta_k \langle A (x^k - x^{k -
    1}), y^k - y^{k + 1} \rangle. \label{main-ineq-last}
  \end{split} \end{align}
  Multiplying both sides of \eqref{main-ineq-2} by $\eta_k$, using the inequality \eqref{main-ineq-last} and the fact that $\eta_k \theta_k =
  \eta_{k - 1}$, we obtain
   \begin{align}\begin{split}
    & \beta_k \eta_k Q (\tilde{z }^{k + 1}, z) - (\beta_k - 1) \eta_k Q
    (\tilde{z }^k, z) \\
    \leq & \frac{\eta_k}{2 \tau_k} \| x - x^k \|^2
    - \frac{\eta_k}{2 \tau_k (1 + p_k)}  \| x - x^{k + 1} \|^2 \label{main-ineq-3} + \frac{\eta_k}{2 \sigma_k} \| y - y^k \|^2 - \frac{\eta_k}{2 \sigma_k}
    \| y - y^{k + 1} \|^2 - \frac{\eta_k}{2 \sigma_k} \| y^{k + 1} - y^k \|^2
     \\
    & - \eta_k \left( \frac{1}{2 \tau_k} - \frac{4(p_k + 1) \gamma_k^{2
    N}}{p_k \tau_k (1 - 8 \gamma_k^{2 N})} - \frac{4\gamma_k^{2 N} \| Q \|}{1 -
    8 \gamma_k^{2 N}} \right) \| x^{k + 1} - x^k \|^2 \\
    & + \eta_k \langle A (x^{k + 1} - x^k), y - y^{k + 1} \rangle - \eta_{k - 1} \langle A (x^k - x^{k - 1}), y - y^k \rangle - \eta_{t -
    1} \langle A (x^k - x^{k - 1}), y^k - y^{k + 1} \rangle.  
  \end{split} \end{align}
  Apply Cauchy-Schwartz inequality to the last term of \eqref{main-ineq-3}
   \begin{align}\begin{split}
    - \eta_{k - 1} \langle A (x^k - x^{k - 1}), y^k - y^{k + 1} \rangle
    \leq & \frac{\| A \|^2 \eta_{k - 1}^2 \sigma_k}{2 \eta_k} \| x^k -
    x^{k - 1} \|^2 + \frac{\eta_k}{2 \sigma_k} \| y^{k + 1} - y^k \|^2
     \\
    = & \frac{\| A \|^2 \eta_{k - 1} \sigma_{k - 1}}{2} \| x^k - x^{k - 1}
    \|^2 + \frac{\eta_k}{2 \sigma_k} \| y^{k + 1} - y^k \|^2 \label{main-ineq-3-last}
  \end{split} \end{align}
  where the last equality holds because $\frac{\eta_k}{\eta_{k + 1}} = \frac{\sigma_k}{\sigma_{k + 1}} = \frac{k + 1}{k + 2}$. Plugging \eqref{main-ineq-3-last} into \eqref{main-ineq-3}, we get 
   \begin{align}\begin{split}  \label{main-ineq-4}
    & (\beta_{k + 1} - 1) \eta_{k + 1} Q (\tilde{z }^{k + 1}, z) - (\beta_k
    - 1) \eta_k Q (\tilde{z }^k, z)\\
    = & \beta_k \eta_k Q (\tilde{z
    }^{k + 1}, z) - (\beta_k - 1) \eta_k Q (\tilde{z }^k, z)\\
    \leq & \frac{\eta_k}{2 \tau_k} \| x - x^k \|^2 - \frac{\eta_{k + 1}}{2
    \tau_{k + 1}}  \| x - x^{k + 1} \|^2 + \frac{\eta_k}{2 \sigma_k} \| y -
    y^k \|^2 - \frac{\eta_{k + 1}}{2 \sigma_{k + 1}} \| y - y^{k + 1} \|^2 \\
    & + \eta_k \langle A (x^{k + 1} - x^k), y - y^{k + 1} \rangle  - \eta_{t -
    1} \langle A (x^k - x^{k - 1}), y - y^k \rangle  \\
    &+ \frac{\| A \|^2 \eta_{k - 1} \sigma_{k - 1}}{2} \| x^k - x^{k - 1} \|^2 - \eta_k \left( \frac{1}{2 \tau_k} - \frac{4(p_k + 1) \gamma_k^{2 N}}{p_k
    \tau_k (1 - 8 \gamma_k^{2 N})} - \frac{4\gamma_k^{2 N} \| Q \|}{1 - 8
    \gamma_k^{2 N}} \right) \| x^{k + 1} - x^k \|^2.  
  \end{split} \end{align}
  Summing up the inequality \eqref{main-ineq-4} from $0$ to $k$ and noting that $\beta_0 = 1$, we get
   \begin{align}\begin{split}
    \beta_k \eta_k Q (\tilde{z }^{k + 1}, z) \leq & \eta_k \langle A
    (x^{k + 1} - x^k), y - y^{k + 1} \rangle + \frac{\eta_0}{2 \tau_0} \| x -
    x^0 \|^2 + \frac{\eta_0}{2 \sigma_0} \| y - y^0 \|^2 - \frac{\eta_{k + 1}}{2 \tau_{k + 1}}  \| x - x^{k + 1} \|^2 \\
    &  -
    \frac{\eta_{k + 1}}{2 \sigma_{k + 1}} \| y - y^{k + 1} \|^2   - \eta_k \left( \frac{1}{2 \tau_k} - \frac{4(p_k + 1) \gamma_k^{2
    N}}{p_k \tau_k (1 - 8 \gamma_k^{2 N})} - \frac{4\gamma_k^{2 N} \| Q \|}{1 -
    8 \gamma_k^{2 N}} \right) \| x^{k + 1} - x^k \|^2  \\
    & - \sum_{i = 0}^{k - 1} \eta_i \left( \frac{1}{2 \tau_i} - \frac{4(p_i +
    1) \gamma_i^{2 N}}{p_i \tau_i (1 - 8 \gamma_i^{2 N})} - \frac{4\gamma_i^{2
    N} \| Q \|}{1 - 8 \gamma_i^{2 N}} - \frac{\| A \|^2 \sigma_i}{2} \right)
    \| x^{i + 1} - x^i \|^2  \\
    \leq & \frac{\eta_0}{2 \tau_0} \| x - x^0 \|^2 + \frac{\eta_0}{2
    \sigma_0} \| y - y^0 \|^2 + \eta_k \langle A (x^{k + 1} - x^k), y - y^{t +
    1} \rangle \label{main-sum-1}\\
    & - \eta_k \left( \frac{1}{2 \tau_k} - \frac{4(p_k + 1) \gamma_k^{2
    N}}{p_k \tau_k (1 - 8 \gamma_k^{2 N})} - \frac{4\gamma_k^{2 N} \| Q \|}{1 -
    9 \gamma_k^{2 N}} \right) \| x^{k + 1} - x^k \|^2 - \frac{\eta_{k + 1}}{2
    \sigma_{k + 1}} \| y - y^{k + 1} \|^2  
  \end{split} \end{align}
  where the last inequality holds because, as we choose $p_k = \frac{1}{k + 1}, \tau_k = \frac{(k + 1)}{2 (\gamma^N_K \| Q \| +
  K \| A \|)}, \sigma_k = \frac{k+1}{2 K \| A \|}$, and $N \geq \log_{\gamma_K} \frac{1}{K}$ or equivalently $\gamma_K^N \leq 1/K$, we have for all $1\leq i \leq K$,
 \begin{align}\begin{split}\label{ass-verify}
    & \frac{1}{2 \tau_i} - \frac{(p_i + 1) \gamma_i^{2 N}}{p_i \tau_i (1 - 2
    \gamma_i^{2 N})} - \frac{\gamma_i^{2 N} \| Q \|}{1 - 2 \gamma_i^{2 N}} -
    \frac{\| A \|^2 \sigma_i}{2}  \\
    = & \frac{\gamma^N_K \| Q \|}{(i + 1)} + \frac{K \| A \|}{(i + 1)} -
    \frac{\gamma_i^{2 N} 2 (\gamma^N_K \| Q \| + K \| A \|)}{(1 - 2
    \gamma_i^{2 N})} - \frac{\gamma_i^{2 N} \| Q \|}{1 - 2 \gamma_i^{2 N}} -
    \frac{\| A \| (i + 1)}{4 K}  \\
    \geq & \left( \frac{K}{(i + 1)} - \frac{2 \gamma_i^{2 N}}{(1 - 2
    \gamma_i^{2 N})} - \frac{i + 1}{4 K} \right) \| A \| + \left( \frac{1}{\gamma^N_K (i + 1)} - \frac{2 \gamma^N_K}{(1 - 2
    \gamma_i^{2 N})} - \frac{1}{1 - 2 \gamma_i^{2 N}} \right) \gamma_i^{2 N}
    \| Q \| \\
    \geq & 0 .
  \end{split} \end{align}
  A similar analysis as \eqref{main-ineq-3-last} shows that 
   \begin{align}\begin{split}\label{ineq:final}
    \eta_k \langle A (x^{k + 1} - x^k), y - y^{k + 1} \rangle 
    \leq & \frac{\| A \|^2 \eta^2_k \sigma_{k + 1}}{2 \eta_{k + 1}} \|
    x^{k + 1} - x^k \|^2 + \frac{\eta_{k + 1}}{2 \sigma_{k + 1}} \| y - y^{t +
    1} \|^2  \\
    = & \frac{\| A \|^2 \eta_k \sigma_k}{2} \| x^{k + 1} - x^k \|^2 +
    \frac{\eta_{k + 1}}{2 \sigma_{k + 1}} \| y - y^{k + 1} \|^2  
  \end{split} \end{align} 
  where the first inequality holds because of the Cauchy-Schwartz inequality and the last inequality holds because of the fact $\frac{\eta_k}{\eta_{k + 1}} = \frac{\sigma_k}{\sigma_{k + 1}} = \frac{k + 1}{k + 2}$. Finally, substitute \eqref{ineq:final} into \eqref{main-sum-1} to get
   \begin{align}\begin{split}
    & \beta_k Q (\tilde{z }^{k + 1}, z) \\
    \leq & \frac{\eta_0}{2 \eta_k
    \tau_0} \| x - x^0 \|^2 + \frac{\eta_0}{2 \eta_k \sigma_0} \| y - y^0 \|^2 - \eta_k \left( \frac{1}{2 \tau_k} - \frac{4(p_k + 1) \gamma_k^{2
    N}}{p_k \tau_k (1 - 8 \gamma_k^{2 N})} - \frac{4\gamma_k^{2 N} \| Q \|}{1 -
    9 \gamma_k^{2 N}} - \frac{\| A \|^2 \sigma_k}{2} \right) \| x^{k + 1} - x^k \|^2
     \\
    \leq  & \frac{1}{2 \tau_k} \| x - x^0 \|^2 + \frac{1}{2 \sigma_k} \| y - y^0
    \|^2  
  \end{split} \end{align}
  which complete the proof of part 1).

  Next, for the proof of part 2), recall the inequality \eqref{main-sum-1}
   \begin{align}\begin{split}
    \beta_k \eta_k Q (\tilde{z }^{k + 1}, z) \leq & \frac{\eta_0}{2
    \tau_0} \| x - x^0 \|^2 + \frac{\eta_0}{2 \sigma_0} \| y - y^0 \|^2
     \\
    & - \frac{\eta_{k + 1}}{2 \tau_{k + 1}}  \| x - x^{k + 1} \|^2 -
    \frac{\eta_{k + 1}}{2 \sigma_{k + 1}} \| y - y^{k + 1} \|^2 + \eta_k
    \langle A (x^{k + 1} - x^k), y - y^{k + 1} \rangle  \\
    & - \eta_k \left( \frac{1}{2 \tau_k} - \frac{4(p_k + 1) \gamma_k^{2
    N}}{p_k \tau_k (1 - 8 \gamma_k^{2 N})} - \frac{4\gamma_k^{2 N} \| Q \|}{1 -
    8 \gamma_k^{2 N}} \right) \| x^{k + 1} - x^k \|^2  \\
    \leq & \frac{\eta_0}{2 \tau_0} \| x - x^0 \|^2 + \frac{\eta_0}{2
    \sigma_0} \| y - y^0 \|^2 - \frac{\eta_{k + 1}}{2 \tau_{k + 1}}  \| x -
    x^{k + 1} \|^2 - \frac{\eta_{k + 1}}{4 \sigma_{k + 1}} \| y - y^{k + 1}
    \|^2  \\
    & - \eta_k \left( \frac{1}{2 \tau_k} - \frac{4(p_k + 1) \gamma_k^{2
    N}}{p_k \tau_k (1 - 8 \gamma_k^{2 N})} - \frac{4\gamma_k^{2 N} \| Q \|}{1 -
    8 \gamma_k^{2 N}} - \| A \|^2 \sigma_k \right) \| x^{k + 1} - x^k \|^2
     \\
    \leq & \frac{\eta_0}{2 \tau_0} \| x - x^0 \|^2 + \frac{\eta_0}{2
    \sigma_0} \| y - y^0 \|^2 - \frac{\eta_{k + 1}}{2 \tau_{k + 1}}  \| x -
    x^{k + 1} \|^2 - \frac{\eta_{k + 1}}{4 \sigma_{k + 1}} \| y - y^{k + 1}
    \|^2.  
  \end{split} \end{align}
  Since we choose $\frac{\eta_0}{\eta_k} = \frac{\sigma_0}{\sigma_k} = \frac{\tau_0}{\tau_k} = \frac{1}{k+1}$ and $\frac{\eta_{k+1}}{\eta_k} = \frac{\sigma_{k+1}}{\sigma_k} = \frac{\tau_{k+1}}{\tau_k} = \frac{k+2}{k+1}$, the above inequality implies that
   \begin{align}\begin{split}
    0 \leq \beta_k Q (\tilde{z }^{k + 1}, z) \leq & \frac{1}{2
    \tau_k} \| x - x^0 \|^2 + \frac{1}{2 \sigma_k} \| y - y^0 \|^2 -
    \frac{1}{2 \tau_k}  \| x - x^{k + 1} \|^2 - \frac{1}{4 \sigma_k} \| y -
    y^{k + 1} \|^2.  
  \end{split} \end{align}
Thus, we complete the proof of part 2).
\end{proof}

\begin{lemma} \label{lemma:fix-2}
  Under the same condition as in \Cref{lemma:fix-1}, we have
  \begin{itemize}
      \item[1)] it holds for any $z \in \mathcal{Z}$ that
  \[ Q (\tilde{z }^K, z) \leq 2 \left( \frac{\gamma^N_K \| Q \|}{K^2}
     + \frac{\| A \|}{K} \right) \| z - z^0 \|^2. \]
      \item[2)] it holds for any $0 \leq t \leq K$ with  $K \geq \frac{\gamma^N_K \| Q \|}{\|
  A \|}$ that
  \[ \frac{1}{2} \| z^k - z^{\ast} \|^2 \leq \| z^0 - z^{\ast} \|^2. \]
  \end{itemize}
\end{lemma}

\begin{proof}
Part 2) of Lemma \ref{lemma:fix-1} implies that
     \begin{align}\begin{split}
    \beta_K Q (\tilde{z }^{K + 1}, z) \leq &  \frac{1}{2 \tau_K} \| x - x^0 \|^2 + \frac{1}{2 \sigma_K} \| y - y^0
    \|^2  \leq \frac{\gamma^N_K \| Q \| + K \| A \|}{K + 1} \| x - x^0 \|^2 + \| A
    \| \| y - y^0 \|^2. 
  \end{split} \end{align}
  Noting that $\beta_K = K + 1$, we get
   \begin{align}\begin{split}
    Q (\tilde{z }^K, z) \leq & \frac{\gamma^N_K \| Q \| + K \| A
    \|}{K^2} \| x - x^0 \|^2 + \frac{\| A \|}{K} \| y - y^0 \|^2
    \leq  2 \left( \frac{\gamma^N_K \| Q \|}{K^2} + \frac{\| A \|}{K}
    \right) \| z - z^0 \|^2.  
  \end{split} \end{align}
  So we complete the proof of part 1). Similarly, based on part 1) of Lemma \ref{lemma:fix-1}, we have
  \[
  \frac{1}{2 \tau_k}  \| x - x^{k + 1} \|^2 + \frac{1}{4 \sigma_k} \| y -
    y^{k + 1} \|^2 \leq \frac{1}{2\tau_k} \| x - x^0 \|^2 + \frac{1}{2 \sigma_k} \| y - y^0 \|^2
  \]
  and
  \[  \| x - x^{k + 1} \|^2 + \frac{\tau_k}{2 \sigma_k} \| y - y^{k + 1} \|^2
     \leq \| x - x^0 \|^2 + \frac{\tau_k}{\sigma_k} \| y - y^0 \|^2. \]
  Then, the inequality $K \geq \frac{\gamma^N_K \| Q \|}{\|
  A \|}$ implies that
   \begin{align}\begin{split}
    \frac{1}{2} \leq \frac{K \| A \|}{\gamma^N_K \| Q \| + K \| A \|}
    \leq \frac{\tau_k}{\sigma_k} = & \frac{K \| A \|}{\gamma^N_K \| Q \|
    + K \| A \|} \leq 1  
  \end{split} \end{align}
  Thus, the following inequality holds for any $0 \leq k
  \leq K - 1$,
  \[ \frac{1}{2} \| z - z^{k + 1} \|^2 \leq \| x - x^{k + 1} \|^2 +
     \frac{\tau_k}{2 \sigma_k} \| y - y^{k + 1} \|^2 \leq \| x - x^0 \|^2
     + \frac{\tau_k}{\sigma_k} \| y - y^0 \|^2 \leq \| z - z^0 \|^2. \]
  By choosing $z = z^{\ast}$, we complete the proof of part 2).
\end{proof}

\begin{lemma} \label{lemma:fix-3}
  Under the same condition as in \Cref{lemma:fix-1}, suppose the total iteration number $K$ satisfies
  \[ K \geq \max \left\{ \frac{4 \| A \|}{\xi}, \sqrt{\frac{4 \gamma^N_K
     \| Q \|}{\xi}} \right\}. \]
  It holds for any $\dot{z} \in \mathcal{Z}$ that
  \[ G_{\xi} (\tilde{z}, \dot{z}) \leq 4 \left( \frac{\gamma^N_K \| Q
     \|}{K^2} + \frac{\| A \|}{K} \right) \| z^0 - \dot{z} \|^2. \]
\end{lemma}
\begin{proof}
  By lemma \ref{lemma:fix-2}, for any $z \in \mathcal{Z}$, we have
  \[ G_{\xi} (\tilde{z}, \dot{z}) =  Q (\tilde{z}, z) - \xi \| z - \dot{z} \|^2 \leq 2 \left(
     \frac{\gamma^N_K \| Q \|}{K^2} + \frac{\| A \|}{K} \right) \| z^0 - z
     \|^2 - \frac{\xi}{2} \| z - \dot{z} \|^2. \]
  With the choice $K \geq \max \left\{ \frac{4 \| A \|}{\xi},
  \sqrt{\frac{4 \gamma^N_K \| Q \|}{\xi}} \right\}$, we have $2 \left(
  \frac{\gamma^N_K \| Q \|}{K^2} + \frac{\| A \|}{K} \right) < \frac{\xi}{2}$,
  thus the right hand side attain its maximum at $z = \frac{\frac{\xi}{2}
  \dot{z} - 2 \left( \frac{\gamma^N_K \| Q \|}{K^2} + \frac{\| A \|}{K}
  \right) z^0}{\frac{\xi}{2} - 2 \left( \frac{\gamma^N_K \| Q \|}{K^2} +
  \frac{\| A \|}{K} \right)}$. Therefore, we have
  \[ G_{\xi} (\tilde{z}, \dot{z}) \leq \frac{2 \xi \left(
     \frac{\gamma^N_K \| Q \|}{K^2} + \frac{\| A \|}{K} \right)}{\xi - 4
     \left( \frac{\gamma^N_K \| Q \|}{K^2} + \frac{\| A \|}{K} \right)} \| z^0
     - \dot{z} \|^2 \leq 4 \left( \frac{\gamma^N_K \| Q
     \|}{K^2} + \frac{\| A \|}{K} \right) \| z^0 - \dot{z} \|^2 \]
  where the last inequality holds because $\xi - 4 \left( \frac{\gamma^N_K \| Q \|}{K^2} + \frac{\| A \|}{K}
  \right) \geq \frac{\xi}{2}$. So we complete the proof.
\end{proof}

Now, we are ready to prove \Cref{thm:fix}.

\begin{proof}[Proof of \Cref{thm:fix}]
  We complete the proof by mathematical induction. The case $n = 0$ is trival. Suppose it holds for $n \leq D$. Denote by $z^{\ast}_n = \tmop{argmin}_{z \in \mathcal{Z}^{\ast}} \| z^{n, 0} -
  z \|^2$ the projection of $z^{n,0}$ onto the optimal set $\mathcal{Z}^*$. 
   \[\begin{aligned}
    \| z^{D + 1, 0} - z^{0, 0} \| \leq & \sum_{n = 0}^D \| z^{n + 1, 0} -
    z^{n, 0} \| \leq \sum_{n = 0}^D \| z^{n + 1, 0} - z^{\ast}_n \| + \|
    z^{\ast}_n - z^{n, 0} \| \leq \sum_{n = 0}^D 3 \| z^{\ast}_n - z^{n,
    0} \|  \\
    = & 3\sum_{n = 0}^D \tmop{dist} (z^{n, 0}, \mathcal{Z}^{\ast})
    \leq 3 \sum_{n = 0}^D e^{- n} \tmop{dist} (z^{0, 0} -
    \mathcal{Z}^{\ast}) \leq \frac{3}{1 - 1 / e} \tmop{dist} (z^{0, 0},
    \mathcal{Z}^{\ast})  
  \end{aligned}\]
  where the first and second inequality hold because of the triangle inequality, the third inequality use part 2) in \Cref{lemma:fix-2}, and the fourth inequality holds because the mathematical induction assume \Cref{thm:fix} holds for all $n \leq D$. This implies that $z^{D + 1, 0}$ is in set $B_R (z^{0, 0})$ and thus
  This implies that $z^{D + 1, 0}$ is in set $B_R (z^{0, 0})$ and thus
  $\alpha_{\xi}$ is a valid constant of quadratic growth for iterates $z^{D +
  1, 0}$.
   \[\begin{aligned}
    \tmop{dist}^2 (z^{D + 1, 0}, \mathcal{Z}^{\ast}) & \leq
    \frac{1}{\alpha_{\xi}} G_{\xi} (z^{D + 1, 0}, z^{\ast}_D) =
    \frac{1}{\alpha_{\xi}} G_{\xi} (z^{D, K}, z^{\ast}_D)  \\
    & \leq \frac{4}{\alpha_{\xi}} \left( \frac{\gamma^N_K \| Q \|}{K^2}
    + \frac{\| A \|}{K} \right) \| z^{D, 0} - z^{\ast}_D \|^2 =
    \frac{4}{\alpha_{\xi}} \left( \frac{\gamma^N_K \| Q \|}{K^2} + \frac{\| A
    \|}{K} \right) \tmop{dist}^2 (z^{D, 0}, \mathcal{Z}^{\ast})  \\
    & \leq e^{- 2} \tmop{dist}^2 (z^{D, 0}, \mathcal{Z}^{\ast})
    \leq e^{- 2 (D + 1)} \tmop{dist}^2 (z^{0, 0}, \mathcal{Z}^{\ast})
  \end{aligned}\]
where the first inequality holds because of the quadratic growth property, the second inequality holds because of \Cref{lemma:fix-3}, and the third inequality holds because of the choice of $K$. So we complete the proof.
\end{proof}







  



\end{document}